\documentclass[11pt]{amsart}

\usepackage{amssymb,amscd,amsmath,hyperref,color,enumerate}
\usepackage[all,cmtip]{xy}
\usepackage{mathrsfs}
\usepackage{lipsum}
\usepackage{nicematrix,tikz, ifthen}
\usepackage{scalerel}

\usepackage[margin=3cm]{geometry}

\newcommand{\N}{{\ensuremath{\mathbb{N}}}}

\newcommand{\C}{{\ensuremath{\mathbb{C}}}}

\usepackage[normalem]{ulem}

\newcommand{\stkout}[1]{\ifmmode\text{\sout{\ensuremath{#1}}}\else\sout{#1}\fi}

\DeclareMathOperator{\supp}{supp}

\DeclareMathOperator{\diag}{diag}
\DeclareMathOperator{\Tr}{Tr}
\DeclareMathOperator{\id}{id}

\DeclareMathOperator{\spn}{span}

\newcommand{\abs}[1]{\ensuremath{ {\left| #1 \right|} }}
\newcommand{\inner}[2]{\ensuremath{ {\left< #1 , #2 \right>} }}
\newcommand{\norm}[1]{\ensuremath{ {\left\| #1 \right\|} }}
\newcommand{\ca}[1]{\ensuremath{\mathcal{#1}}}

\newtheorem{proposition}{Proposition}[section]
\newtheorem{lemma}[proposition]{Lemma}

\newtheorem{theorem}[proposition]{Theorem}
\newtheorem{corollary}[proposition]{Corollary}
\newtheorem{claim}{Claim}
\newtheorem{case}{Case}
\theoremstyle{definition}

\newtheorem{example}[proposition]{Example}

\newtheorem{remark}[proposition]{Remark}

\numberwithin{equation}{section}

\newlength{\leftstackrelawd}
\newlength{\leftstackrelbwd}
\def\leftstackrel#1#2{\settowidth{\leftstackrelawd}%
	{${{}^{#1}}$}\settowidth{\leftstackrelbwd}{$#2$}%
	\addtolength{\leftstackrelawd}{-\leftstackrelbwd}%
	\leavevmode\ifthenelse{\lengthtest{\leftstackrelawd>0pt}}%
	{\kern-.5\leftstackrelawd}{}\mathrel{\mathop{#2}\limits^{#1}}}

\newcommand{\tripprox}{\mathrel{\setbox0\hbox{$\approx$}%
		\mbox{\makebox[0pt][l]{\raisebox{0.48\ht0}{$\approx$}}$\approx$}}}

\begin{document}
	
	\title[Jordan embeddings and rank preserves of SMAs]{Jordan embeddings and linear rank preservers of structural matrix algebras}
	
	\author{Ilja Gogi\'{c}, Mateo Toma\v{s}evi\'{c}}
	
	\address{I.~Gogi\'c, Department of Mathematics, Faculty of Science, University of Zagreb, Bijeni\v{c}ka 30, 10000 Zagreb, Croatia}
	\email{ilja@math.hr}
	
	\address{M.~Toma\v{s}evi\'c, Department of Mathematics, Faculty of Science, University of Zagreb, Bijeni\v{c}ka 30, 10000 Zagreb, Croatia}
	\email{mateo.tomasevic@math.hr}

	\thanks{We thank the referee for the careful reading of the paper and for useful comments and suggestions.}


	
	\date{\today}

	\begin{abstract}
		We consider subalgebras $\mathcal{A}$ of the algebra $M_n$ of $n \times n$ complex matrices that contain all diagonal matrices, known in the literature as the structural matrix algebras (SMAs). 
		
		Let $\mathcal{A}  \subseteq M_n$ be an arbitrary SMA. We first show that any commuting family of diagonalizable matrices in $\mathcal{A}$ can be intrinsically simultaneously diagonalized (i.e.\ the corresponding similarity can be chosen from $\mathcal{A}$). Using this, we then characterize when one SMA Jordan-embeds into another and in that case we describe the form of such Jordan embeddings. As a consequence, we obtain a description of Jordan automorphisms of SMAs, generalizing Coelho's result on their algebra automorphisms. Next, motivated by the results of Marcus-Moyls and Molnar-\v{S}emrl, connecting the linear rank-one preservers with Jordan embeddings $M_n \to M_n$ and $\mathcal{T}_n \to M_n$ (where $\mathcal{T}_n$ is the algebra of $n \times n$ upper-triangular matrices) respectively, we show that any linear unital rank-one preserver $\mathcal{A} \to M_n$ is necessarily a Jordan embedding. As the converse fails in general, we also provide a necessary and sufficient condition for when it does hold true. Finally, we obtain a complete description of linear rank preservers $\mathcal{A} \to M_n$, as maps of the form $X\mapsto S\left(PX + (I-P)X^t\right)T$, for some invertible matrices $S,T \in M_n$ and a central idempotent $P\in\mathcal{A}$.
	\end{abstract}
	
	\maketitle
	
	\setlength{\parindent}{10pt}

	\section{Introduction}
	
	\subsection{Linear preservers} Linear preserver theory is a very active research area in linear algebra and functional analysis. It considers linear maps between certain matrix or operator spaces which leave some of the relevant properties invariant. The objective is to give an explicit general form of such maps or a sufficiently elegant characterization. There exist many surveys of linear preservers which showcase the theory's history as well as some more recent developments (for instance see \cite{LiPierce, LiTsing}). The theory started with the result of Frobenius \cite{Frobenius} from 1897 which completely describes linear determinant preservers of the algebra $M_n$ of $n \times n$ complex matrices:
	\begin{theorem}[Frobenius]
		Let $\phi : M_n \to M_n$ be a linear map which preserves the determinant, i.e.\ $\det \phi(X)=\det X$ for all $X \in M_n$. Then there exist matrices $A,B \in M_n$ satisfying $\det(AB) = 1$ such that
		\begin{equation}\label{eq:AXB}
			\phi = A(\cdot)B, \qquad \text{ or }\qquad \phi = A(\cdot)^tB.
		\end{equation}
	\end{theorem}
	Starting with the above seminal theorem, linear preserver theory distinguishes itself by especially elegant results. Furthermore, rather deep results are sometimes proved by surprisingly elementary techniques which are cleverly combined and modified depending on the particularities of each problem. Linear preservers also have notable applications even outside pure mathematics. For instance, Wigner's theorem (see e.g.\ \cite{Semrl3,Wigner}) and related results form a cornerstone of the mathematical foundation of quantum mechanics.
	
	To further motivate our present research, we also showcase few more simple linear preserver problems on $M_n$, all taken from the survey paper \cite{LiPierce}, with the corresponding terminology.
	
	\begin{theorem}
		\begin{enumerate}[(a)]
			\item Let $\phi : M_n \to M_n$ be a bijective linear singularity preserver. Then there exist invertible matrices $A,B \in M_n$ such that $\phi$ is of the form $$\phi = A(\cdot)B, \qquad \text{ or }\qquad \phi = A(\cdot)^tB.$$
			\item Let $\phi : M_n \to M_n$ be a linear spectrum preserver. Then there exists an invertible matrix $A \in M_n$ such that $\phi$ is of the form $$\phi = A(\cdot)A^{-1}, \qquad \text{ or }\qquad \phi = A(\cdot)^tA^{-1}.$$
			\item Let $\phi : M_n \to M_n$ be a linear similarity preserver. Then there exist an invertible matrix $A \in M_n$ and $\alpha,\beta \in \C$ such that
			$$\phi = \alpha A(\cdot)A^{-1}+ \beta(\Tr(\cdot))I, \qquad \text{ or }\qquad \phi = \alpha A(\cdot)^t A^{-1}+ \beta(\Tr(\cdot))I$$
			or there exists a fixed $B \in M_n$ such that $\phi = (\Tr(\cdot))B$.
			\item Let $\phi : M_n \to M_n$ be a linear map which preserves unitary matrices or the spectral norm. Then there exist unitary matrices $U,V \in M_n$ such that
			$$\phi = U(\cdot)V, \qquad \text{ or }\qquad \phi = U(\cdot)^tV.$$
		\end{enumerate}
	\end{theorem}
	
	One easily observes in the above examples, as in many others, the appearance of maps of the form \eqref{eq:AXB}, or even more specifically when $B=A^{-1}$, i.e.
	\begin{equation}\label{eq:Jordan homo on Mn}
		\phi = A(\cdot)A^{-1}, \qquad \text{ or }\qquad \phi = A(\cdot)^tA^{-1}.
	\end{equation}
	As a special case of the Skolem-Noether theorem, one concludes that the maps of the form \eqref{eq:Jordan homo on Mn} constitute precisely the automorphisms and antiautomorphisms of the algebra $M_n$, respectively (for a simple proof see \cite[Theorem 1.1 and Corollary 1.3]{Semrl2}). One can unify both types of maps under the more general concept of linear maps $\phi : M_n \to M_n$ which preserve squares, known in the literature as Jordan homomorphisms.
	
	\subsection{Jordan homomorphisms}
	
	If $\ca{A}$ and $\ca{B}$ are rings (associative algebras), a \emph{Jordan homomorphism} is an additive (linear) map $\phi : \ca{A} \to \ca{B}$ such that 
	$$\phi(ab+ba) = \phi(a)\phi(b) + \phi(b)\phi(a), \qquad \text{ for all }a,b \in \ca{A}.$$ 
	When the rings (algebras) are $2$-torsion-free, this is equivalent to
	$$\phi(a^2) = \phi(a)^2, \qquad \text{ for all }a \in \ca{A}.$$
	Jordan homomorphisms actually originate as homomorphisms in the category of Jordan algebras, which is a class of nonassociative algebras analogous to Lie algebras. Jordan algebras were first introduced by Pascual Jordan in 1933 in the context of quantum mechanics \cite{Jordan}. They relate to the standard associative algebras, as the majority of the practically relevant Jordan algebras naturally arise as subalgebras of an associative algebra $\ca{A}$ under a symmetric product given by $x \circ y = xy + yx.$ Reading this, one is immediately reminded of the fact that Lie algebras always embed into an associative algebra under an antisymmetric product $[x,y] = xy - yx$. In contrast, there are Jordan algebras not arising in this way, known in the literature as the \emph{exceptional Jordan algebras} \cite{Albert1}.
	
	\smallskip
	
	In any case, Jordan algebras give rise to the study of Jordan homomorphisms in the context of associative rings and algebras. Ring (algebra) homomorphisms and antihomomorphisms are immediate examples of such maps.
	
	\smallskip
	
	One of the central issues in Jordan theory is to determine the conditions on rings (algebras) $\ca{A}$ and $\ca{B}$ which guarantee that every Jordan homomorphism $\phi : \ca{A} \to \ca{B}$ (possibly under some additional constraints such as surjectivity) is either multiplicative or antimultiplicative. More broadly, the problem involves representing all such Jordan homomorphisms as an appropriate combination of ring (algebra) homomorphisms and antihomomorphisms, if possible. This question has a long and extensive history: in 1950, Jacobson and Rickart \cite{JacobsonRickart} demonstrated that a Jordan homomorphism from any ring to an integral domain is either a homomorphism or an antihomomorphism. This paper is particularly significant for our discussion, as it also shows that for a unital ring $\ca{R}$, any  Jordan homomorphism  from the ring $M_n(\ca{R})$ of $n \times n$ matrices ($n \ge 2$) with entries in $\ca{R}$ to an arbitrary ring $\ca{S}$ can be expressed as a sum of a homomorphism and an antihomomorphism. Similarly, Herstein \cite{Herstein} established that a Jordan homomorphism onto a prime ring is either multiplicative or antimultiplicative. Smiley later refined the same result \cite{Smiley}.
	
	By integrating Herstein's result with the already mentioned fact that all automorphisms of $M_n$ are inner, one can conclude that all nonzero Jordan endomorphisms $\phi$ of $M_n$ are precisely maps of the form \eqref{eq:Jordan homo on Mn} for some invertible matrix $A \in M_n$.
	
	There have been numerous efforts to characterize Jordan homomorphisms, particularly in the context of matrix algebras, through preserver properties. These efforts trace back to at least 1970 and Kaplansky's notable problem \cite{Kaplansky} which asks for conditions on unital (complex) Banach algebras $\ca{A}$ and $\ca{B}$ such that a linear unital invertibility preserver $\phi : \ca{A} \to \ca{B}$ (possibly with some extra conditions, such as surjectivity) is necessarily a Jordan homomorphism. While significant attention has been given to this problem and progress has been made in specific cases, it remains largely unsolved, even for $C^*$-algebras (see e.g.\ \cite[p.\ 270]{BresarSemrl}). 
	
	When considering Jordan homomorphisms of matrix algebras other than the full matrix algebra $M_n$, the first result we mention comes from the paper \cite{MolnarSemrl}, which incidentally also serves as the main motivation for the second part of this paper, which we discuss further on. It concerns the algebra $\ca{T}_n\subseteq M_n$ of all upper-triangular matrices:
	
	\begin{theorem}\cite[Corollary 4]{MolnarSemrl}\label{thm:automorphisms of Tn}
		Let $\phi : \ca{T}_n \to \ca{T}_n$ be a Jordan automorphism. Then there exists an invertible matrix $T \in \ca{T}_n$ such that
		$$\phi = T(\cdot)T^{-1}, \qquad \text{ or }\qquad \phi = (TJ)(\cdot)^{t}(TJ)^{-1},$$
		where $J:= \sum_{i=1}^n E_{i, n+1-i}\in M_n$.
	\end{theorem}
	As the matrix units $E_{ij}$ (especially the diagonal ones) play an essential role in the results of this type, it is natural to consider Jordan embeddings $\phi : \ca{A} \to M_n$, where  $\ca{A}$ is a unital subalgebra of $M_n$ spanned by a subset of matrix units, and characterizations of such maps via preserving properties. Such algebras $\ca{A}$ were introduced by van Wyk in 1988 as \emph{structural matrix algebras} (in short, SMAs)  \cite{VanWyk1}. In fact, SMAs are precisely the subalgebras of $M_n$ which contain all diagonal matrices (Proposition \ref{prop:SMAschar}). SMAs have been studied in, for example, \cite{Akkurt, Akkurt2, AkkurtBarkerWild, BeslagaDascalescu, BeslagaDascalescu2, Coelho, Coelho2,VanWyk2}. Work has also been done on \emph{incidence algebras}, which are similarly defined but somewhat more general, venturing into infinite-dimensional settings (see \cite{BrusamarelloFornaroliKhrypchenko1, BrusamarelloFornaroliKhrypchenko2, GarcesKhrypchenko, GarcesKhrypchenko2, SlowikVanWyk}).
	
	Coelho's important paper \cite{Coelho} gives a complete description of (algebra) automorphisms of SMAs (\cite[Theorem C]{Coelho}), as well as supplying necessary and sufficient conditions on the structure of SMA such that all automorphisms are inner (see Theorem \ref{thm:when are all automorphism inner?} further in the text).  The paper \cite{Coelho} cleverly combines the methods of abstract algebra with the purely combinatorial techniques of graph theory. It serves as both model and the main inspiration for our first main result, Theorem \ref{thm:general form of phi}, and its immediate consequence, Corollary \ref{cor:SMA Jordan embeddings}. We also mention the paper \cite{Akkurt} by Akkurt et.\ al.\ which establishes the aforementioned automorphism description \cite[Theorem C]{Coelho} using alternative methods, which themselves proved useful for our results. We reference this reformulated result as Theorem \ref{thm:description of Aut(A_rho)}.
	
	Regarding Jordan homomorphisms of SMAs, the most important result is again due to Akkurt et.\ al. Namely, \cite[Case 2]{Akkurt2} states that for suitable SMAs $\ca{A}$, any Jordan homomorphism $\phi : \ca{A} \to \ca{B}$, where $\ca{B}$ is an arbitrary ring, is necessarily a sum of a homomorphism and an antihomomorphism. Futhermore, \cite[Case 1]{Akkurt2} extends the previous result of Benkovi\v{c} \cite{Benkovic}, which describes Jordan homomorphisms of triangular algebras. Another relevant paper is \cite{GarcesKhrypchenko}. It considers incidence algebras over finite connected posets and formulates complete descriptions of their Jordan automorphisms (viz.\  \cite[Section 4]{GarcesKhrypchenko}), along with an elegant characterization as (idemp)potent preservers.

	\subsection{Rank and rank-one preservers}
	
	Now let us circle back to preserver problems, more specifically to the problem of linear rank preservers. We start by stating the celebrated result of Marcus and Moyls from 1959:
	\begin{theorem}\cite[Theorem 1]{MarcusMoyls}\label{thm:rank preserver on Mn}
		Let $\phi : M_n \to M_n$ be a linear rank preserver. Then there exist invertible matrices $A,B \in M_n$ such that $\phi$ is of the form \eqref{eq:AXB}.
		
	\end{theorem}
	
	The same authors also showed that it actually suffices to assume that $\phi$ is merely a rank-one preserver to obtain the same conclusion:
	
	\begin{theorem}\cite[Corollary, p.\ 1219]{MarcusMoyls2}
		Let $\phi : M_n \to M_n$ be a linear rank-one preserver. Then there exist invertible matrices $A,B \in M_n$ such that $\phi$ is of the form \eqref{eq:AXB}.
	\end{theorem}
	
	In particular, the unital linear rank (or rank-one) preservers $\phi : M_n \to M_n$ are precisely the Jordan automorphisms of $M_n$. Analogous linear preserver problems were also considered for the upper-triangular algebra $\ca{T}_n$. We already mentioned the important paper \cite{MolnarSemrl} by  Molnar and \v{S}emrl in the context of Jordan automorphisms of $\ca{T}_n$. In addition, the same paper (see also the related paper \cite{ChooiLim}) presents the following result in the context of rank-one preservers, which we paraphrase here:
	
	\begin{theorem}\cite[Theorem 1]{MolnarSemrl}\label{thm:rank-one preserver on Tn}
		Let $\phi : \ca{T}_n \to M_n$ be a linear rank-one preserver. Then there exist invertible matrices $A,B \in M_n$ such that $\phi$ is of the form \eqref{eq:AXB}  or the image of $\phi$ consists of matrices of rank at most one. 
	\end{theorem}
	
	In particular, a linear rank-one preserver $\phi : \ca{T}_n \to M_n$ is not necessarily a rank preserver, as was the case for maps $\phi : M_n \to M_n$. However, this conclusion does follow if we additionally assume $\phi$ to be, for instance, injective or unital. Knowing this and the structure of Jordan embeddings of SMAs, our  next goal is to formulate a similar result regarding linear unital rank or rank-one preservers $\phi : \ca{A} \to M_n$ where $\ca{A} \subseteq M_n$ is an SMA. In the case of rank-one preservers, unitality indeed turned out to be an indispensable assumption, since without it they can behave rather strangely. For instance, consider the algebra 
	$$
	\ca{A}=\begin{bmatrix}
		* & 0 & * \\
		0 & * & * \\
		0 & 0 & *
	\end{bmatrix} \subseteq M_3
	$$
	and the map
	$$\phi : \ca{A} \to \ca{A}, \qquad \phi\left(\begin{bmatrix}
		x_{11} & 0 & x_{13} \\
		0 & x_{22} & x_{23} \\
		0 & 0 & x_{33}
	\end{bmatrix}\right) = \begin{bmatrix}
		0 & 0 & x_{13} \\
		0 & x_{22} & x_{23} \\
		0 & 0 & x_{11}+x_{33}
	\end{bmatrix}.$$
	Then $\phi$ preserves rank-one matrices, but $\phi(I)$ is a rank-two matrix. Hence this map does not fit into any of the conclusions of Theorem \ref{thm:rank-one preserver on Tn}. Assuming unitality, however, we obtained Theorem \ref{thm:rank-one preserver} (a), which shows that for an arbitrary SMA $\ca{A}\subseteq M_n$ all linear unital rank-one preservers $\phi : \ca{A} \to M_n$ are necessarily Jordan embeddings. While the converse fails in general (Remark \ref{rem:nontrivial transitive map}), in Theorem \ref{thm:rank-one preserver} (b) we managed to characterize when it does hold true by an additional algebraic condition.
	
	We finish the paper with Theorem \ref{thm:rank preserver} and Corollary \ref{cor:rank preserver without unitality}, which we consider to be the main results of the paper. They provide a complete description of linear rank preservers $ \phi : \ca{A} \to M_n$ (where $\ca{A}\subseteq M_n$ is an SMA) and place them in the broader context of Jordan embeddings which were described before. The proof of this result builds heavily on the rank-one result and hence fits in nicely with the first part of the paper.

	\subsection{Paper outline}
	
	This paper is organized as follows. We begin Section \S\ref{sec:preliminaries} by providing relevant terminology and notation. In Section \S\ref{sec:SMAs}
	we present the basic properties of SMAs, as well as our (preparatory) result regarding the intrinsic diagonalization within SMAs (Theorem \ref{thm:inner diagonalization on SMA}). In Section \S\ref{sec:MAMP}, after laying out some results from the literature concerning the (algebra) automorphisms of SMAs (done by Coelho and Akkurt et al.), we prove our first main result, Theorem \ref{thm:general form of phi}, and its consequences regarding Jordan embeddings between two SMAs, as well as their Jordan automorphisms (Corollaries \ref{cor:SMA Jordan embeddings} and \ref{cor:JAUT-SMA}). Section \S\ref{sec:rank-one preserver} presents an intermediary result, Theorem \ref{thm:rank-one preserver}, showing that all unital linear rank-one preservers on SMAs are necessarily Jordan embeddings. This section contains the bulk of the work of the entire paper, since it also lays the ground for the next section.  Finally, in Section \S\ref{sec:rank preserver} we prove our main result Theorem \ref{thm:rank preserver}, along with its immediate consequences.
	
	\section{Preliminaries}\label{sec:preliminaries}
	
	
	We begin this section by introducing some general notation and terminology. First of all, for every set $S$, by $\abs{S}$ we denote its cardinality. If $\rho \subseteq S\times S$ is a binary relation on a set $S$, for a fixed $x \in S$ by $\rho(x)$ and $\rho^{-1}(x)$ we denote its image and preimage by $\rho$, respectively, i.e.
	$$\rho(x)= \{y \in S : (x,y) \in \rho\}, \qquad \rho^{-1}(x)= \{y \in S : (y,x) \in \rho\}.$$
	Throughout this paper $\ca{A}$ will denote a unital complex algebra. By $\ca{A}^\times$ we denote the set of all invertible elements in $\ca{A}$. As usual, $Z(\ca{A})$ stands for the centre of $\ca{A}$. Further, by $\C[x]$ we denote the polynomial algebra in one variable $x$  over $\C$. For $a \in \ca{A}$, by $\C[a]$ we denote the unital subalgebra of $\ca{A}$ generated by $a$, i.e. 
	$$
	\C[a]=\{p(a) : p \in \C[x]\} \subseteq \ca{A}.
	$$ 
	Now let $n \in\N$.
	\begin{itemize}
		\item We denote by $\Delta_n$ the diagonal relation $\{(i,i) :1 \le i \le n\}$ on $\{1,\ldots,n\}$.
		\item As usual, by $M_n := M_n(\C)$ we denote the algebra of all $n\times n$ complex matrices and by $M_{m,n}$ ($m \in \N$) the space of all $m\times n$ complex matrices.
		\item $\ca{T}_n$ and $\ca{D}_n$ denote the subalgebras of all upper-triangular and diagonal matrices of $M_n$, respectively.
		\item  Following \cite{GogicPetekTomasevic}, for $p, k_1, \ldots, k_p \in \N$ such that $k_1+\cdots+k_p = n$,   $\ca{A}_{k_1, \ldots, k_p}$ denotes the corresponding \emph{block upper-triangular subalgebra} of $M_n$, i.e.
		\begin{equation}\label{eq:parabolicalg}
			\ca{A}_{k_1, \ldots, k_p}:=\begin{bmatrix} M_{k_1,k_1} & M_{k_1,k_2} & \cdots & M_{k_1,k_p} \\ 0 & M_{k_2,k_2} & \cdots & M_{k_2,k_p} \\
				\vdots & \vdots & \ddots & \vdots \\
				0 & 0 & \cdots & M_{k_p,k_p} \end{bmatrix}.
		\end{equation} 
		\item For $A,B \in M_n$ we denote their Jordan product by $A \circ B := AB + BA$. 
		\item For $A,B \in M_n$, by $A \leftrightarrow B$ we denote that
		$A$ and $B$ commute, i.e.\ $AB = BA$. The notion of a matrix commuting with a subset of $M_n$ will be denoted in the same way.
		\item For $A,B \in M_n$ we say that $A$ and $B$ are \emph{orthogonal} if $AB = BA = 0$.
		\item For $A \in M_n$, by $r(A)$ and $\sigma(A)$ we denote the rank and the spectrum of $A$, respectively.
		\item We denote by $\diag(\lambda_1,\ldots,\lambda_n) \in \ca{D}_n$ the diagonal matrix with diagonal entries equal to $\lambda_1,\ldots,\lambda_n\in \C$ (in this order).
		\item For $1 \le i,j \le n$ we denote by $E_{ij}\in M_n$ the standard matrix unit with $1$ at the position $(i,j)$ and $0$ elsewhere. Similarly, the canonical basis vectors of $\C^n$ are denoted by $e_1, \ldots, e_n$.
		\item For a subset $S$ of $\C^n$ or $M_n$, by $\spn S$ we denote its linear span.
		\item For vectors $u,v \in \C^n$ by $u \parallel v$ we denote  that the set $\{u,v\}$ is linearly dependent. The same notation is used for matrices.
		\item By $\inner{u}{v} = v^*u$ we denote the standard inner product of vectors $u,v \in \C^n$.
		\item For any permutation $\pi \in S_n$ (where, as usual, $S_n$ denotes the symmetric group), by \begin{equation}\label{eq:definition of permutation matrix}
			R_{\pi} := \sum_{k=1}^n E_{k\pi(k)}
		\end{equation} we denote the permutation matrix in $M_n$ associated to $\pi$.
	\end{itemize}
	
	As any matrix $A = [A_{ij}]_{i,j=1}^n \in M_n$ can be understood as a map $\{1,\ldots, n\}^2 \to \C, (i,j) \mapsto A_{ij}$, we can consider its \emph{support }$\supp A$ as the set of all indices $(i,j) \in \{1,\ldots, n\}^2$ such that $A_{ij} \ne 0$. Moreover, for a set $S \subseteq \{1,\ldots,n\}^2$ we also say that the matrix $A$ is \emph{supported in $S$} if $\supp A \subseteq S$.
	
	\smallskip
	
	We also state a few basic properties of general Jordan homomorphisms of complex algebras, proofs of which are elementary and can be found in \cite{JacobsonRickart}. 
	
	\begin{lemma}\label{le:Jordan homomorphism basic properties}
		Let $\phi : \ca{A} \to \ca{B}$ be a Jordan homomorphism between algebras $\ca{A}$ and $\ca{B}$. We have:
		\begin{enumerate}[(a)]
			\item $\phi(aba) = \phi(a)\phi(b)\phi(a)$  for all $a,b \in \ca{A}$.
			\item $\phi(abc + cba) = \phi(a)\phi(b)\phi(c) + \phi(c)\phi(b)\phi(a)$ for all $a,b,c \in \ca{A}$.
			\item For every idempotent $p \in \ca{A}$ and $a \in \ca{A}$ which commutes with $p$ we have $\phi(pa) = \phi(p)\phi(a) = \phi(a)\phi(p)$.
		\end{enumerate}
	\end{lemma}
	As usual, by a Jordan embedding $\phi : \ca{A} \to \ca{B}$, we mean an injective Jordan homomorphism (i.e.\ a monomorphism).

	\section{Structural matrix algebras}\label{sec:SMAs}
	\subsection{Definition and basic properties}
	By a \emph{quasi-order} on $\{1,\ldots,n\}$ we mean a reflexive transitive relation $\rho$ on $\{1,\ldots,n\}$. For a quasi-order $\rho$ we define the subspace of $M_n$ by
	$$\ca{A}_\rho :=\{A \in M_n : \supp A \subseteq \rho\}=\spn\{E_{ij} : (i,j) \in \rho\}.$$
	It is easy to see that $\ca{A}_\rho$ is in fact a unital subalgebra of $M_n$. Following \cite{VanWyk1}, we refer to $\ca{A}_\rho$ as a \emph{structural matrix algebra (SMA) defined by the quasi-order $\rho$}. Throughout the paper we also use the following auxiliary relations: 
	$$
	\rho^\times:= \rho\setminus \Delta_n, \qquad \rho^t := \{(y,x) : (x,y) \in \rho\}.
	$$
	Obviously, $\rho^t$ is also a quasi-order on $\{1,\ldots,n\}$, which we refer to as the \emph{reverse quasi-order} of $\rho$.
	
	\smallskip 
	
	SMAs have a simple characterization within the class of all unital subalgebras of $M_n$:
	
	\begin{proposition}\label{prop:SMAschar}
		Let $\ca{A} \subseteq M_n$ be a unital subalgebra. The following statements are equivalent:
		\begin{enumerate}[(i)]
			\item $\ca{A}$ is an SMA.
			\item $\ca{A}$ contains all diagonal matrices.
			\item $\ca{A}$ contains a diagonal matrix with distinct diagonal entries.
		\end{enumerate}
	\end{proposition}
	\begin{proof}
		The implications $(i) \implies (ii) \implies (iii)$ are trivial.
		
		\fbox{$(iii) \implies (ii)$} Suppose that $D \in \ca{D}_n \cap \ca{A}$ has distinct diagonal entries. Since the minimal polynomial of $D$ has degree $n$, it follows that the set $\{I,D, \ldots, D^{n-1}\} \subseteq \ca{A}$ is linearly independent, and hence a basis for $\ca{D}_n$.  Thus, $\ca{D}_n \subseteq \ca{A}$.
		
		\smallskip 
		
		\fbox{$(ii) \implies (i)$} Suppose that $\ca{D}_n \subseteq \ca{A}$. Define a subset $\rho \subseteq \{1,\ldots,n\}^2$ as
		$$\rho := \bigcup_{A \in \ca{A}} \supp A.$$
		One easily shows that $\rho$ is a quasi-order. We claim that $\ca{A}=\ca{A}_\rho$. Indeed, let $(i,j) \in \rho$ be arbitrary. By definition, there exists $A \in \ca{A}$ such that $A_{ij} \ne 0$. We have
		$$\ca{A} \ni E_{ii}AE_{jj} = A_{ij}E_{ij} \implies E_{ij} \in \ca{A}.$$
		We conclude that $\ca{A}_{\rho}\subseteq \ca{A}$. On the other hand, for any $A \in \ca{A}$ we have $\supp A \subseteq \rho$ and hence $A\in \ca{A}_{\rho}$. Therefore, $\ca{A} = \ca{A}_{\rho}$ is an SMA.
		
	\end{proof}
	
	Let $\rho, \rho' \subseteq \{1,\ldots,n\}^2$ be quasi-orders. A permutation $\pi \in S_n$ is said to be a \emph{$(\rho,\rho')$-increasing} if 
	$$
	(\pi(i), \pi(j)) \in \rho', \qquad \text{ for all } (i,j) \in \rho.
	$$
	Such a permutation $\pi$ gives rise  to an algebra embedding
	\begin{equation}\label{eq:conjugation by permutation}
		\ca{A}_\rho \to \ca{A}_{\rho'}, \qquad E_{ij} \mapsto E_{\pi(i)\pi(j)}, \qquad (i,j) \in \rho.
	\end{equation}
	This map can be expressed precisely as $R_{\pi}(\cdot)R_{\pi}^{-1}$ where $R_{\pi} \in M_n$ is the associated permutation matrix \eqref{eq:definition of permutation matrix}.
	Moreover, for any permutation $\pi \in S_n$, from \eqref{eq:conjugation by permutation} it is clear that the automorphism $R_{\pi}(\cdot)R_{\pi}^{-1}$ of $M_n$ restricts to an embedding $\ca{A}_\rho \to \ca{A}_{\rho'}$ if and only if $\pi$ is $(\rho,\rho')$-increasing.
	
	Furthermore, if $\abs{\rho} = \abs{\rho'}$, which is equivalent to $\dim \ca{A}_{\rho} = \dim \ca{A}_{\rho'}$, a $(\rho,\rho')$-increasing permutation $\pi \in S_n$ is in fact a \emph{quasi-order isomorphism}, i.e.
	$$(i,j)\in \rho \iff (\pi(i),\pi(j))\in \rho', \qquad \text{ for all }1\leq i,j \leq n.$$
	Indeed, the above is equivalent to $R_{\pi} \ca{A}_{\rho} R_{\pi}^{-1} = \ca{A}_{\rho'}$, which follows from $R_{\pi} \ca{A}_{\rho} R_{\pi}^{-1} \subseteq \ca{A}_{\rho'}$ and the equality of dimensions.
	
	\smallskip
	
	\noindent Next, given a quasi-order $\rho$, we define the equivalence relation $\overline{\rho}$ on the set $\{1,\ldots, n\}$ as
	$$(i,j) \in \overline{\rho} \stackrel{\text{def}}\iff (i,j),(j,i) \in \rho.$$
	We provide an explicit argument for the following fact, which can be deduced from {\cite[p.\ 432]{Akkurt}}.
	
	\begin{lemma}\label{le:Akkurt permutation argument}
		Let $\mathcal{A}_\rho \subseteq M_n$ be an SMA. There exists a permutation $\pi \in S_n$ such that
		\begin{equation}\label{eq:Akkurt permutation argument}
			R_{\pi} \ca{A}_\rho R_{\pi}^{-1}=\begin{bmatrix} M_{m_1,m_1} & M_{m_1,m_2}(\mathbb{K}) & \cdots & M_{m_1,m_p}(\mathbb{K}) \\ 0 & M_{m_2,m_2} & \cdots & M_{m_2,m_p}(\mathbb{K}) \\
				\vdots & \vdots & \ddots & \vdots \\
				0 & 0 & \cdots & M_{m_p,m_p} \end{bmatrix}
		\end{equation}
		for some $p, m_1,\ldots,m_p \in \N$ such that $m_1+\cdots+m_p = n$, where for any $1\leq i < j \leq n$, $M_{m_i,m_j}(\mathbb{K})$ is either zero or $M_{m_i,m_j}$. 
	\end{lemma}
	\begin{proof}
		On the quotient set $\{1,\ldots,n\}/\overline{\rho}$ we define the relation $\preceq$ in the following way: for $1 \le i,j \le n$ the corresponding $\overline{\rho}$-classes satisfy 
		$$
		[i]_{\overline{\rho}} \preceq [j]_{\overline{\rho}} \stackrel{\text{def}}{\iff}  (i,j) \in \rho.
		$$ Note that $\preceq$ is well-defined. Indeed, suppose that $i' \in [i]_{\overline{\rho}}$ and $j' \in [j]_{\overline{\rho}}$ are arbitrary. Then we claim that $$(i,j) \in \rho \iff (i',j') \in \rho.$$
		Supposing the former, by definition of $\overline{\rho}$ we obtain
		$$(i',i),(i,j),(j,j') \in \rho \implies (i',j') \in \rho.$$
		The converse is similar. Further, one easily checks  that $\preceq$ is in fact a partial order on the quotient set $\{1,\ldots,n\}/\overline{\rho}$. Now let $$\{1,\ldots,n\}/\overline{\rho} = \{[r_1]_{\overline{\rho}}, \ldots, [r_p]_{\overline{\rho}}\}, \qquad r_1, \ldots, r_p \in \{1,\ldots,n\}$$
		be an ordering of the quotient set which respects the partial order $\preceq$ in the sense that $[r_i]_{\overline{\rho}} \preceq [r_j]_{\overline{\rho}}$ implies $i \le j$ (its existence can be easily shown by an inductive argument).
		
		For each $1 \le k \le p$ denote the elements of the  class $[r_k]_{\overline{\rho}}$ explicitly as
		$$[r_k]_{\overline{\rho}} = \{r_{k,1}, \ldots, r_{k,m_k}\}$$
		for some $m_k \in \N$. Define a permutation $\pi \in S_n$ by
		$$\pi(r_{k,j}) := m_1+\cdots + m_{k-1} + j, \qquad \text{ for }1 \le k \le p \text{ and }1 \le j \le m_k$$
		(here we formally set $m_0 := 0$). Define a new quasi-order $\rho'$ on $\{1,\ldots,n\}$ with the following condition:
		$$(i,j) \in \rho' \stackrel{\text{def}}\iff (\pi^{-1}(i), \pi^{-1}(j)) \in \rho.$$
		The permutation $\pi$ is $(\rho,\rho')$-increasing by definition, so clearly   $\ca{A}_{\rho'}=R_{\pi} \ca{A}_\rho R_{\pi}^{-1}$ (as $|\rho'|=|\rho|$). We now prove \eqref{eq:Akkurt permutation argument}, which is equivalent to 
		\begin{equation}\label{eq:Akkurt permutation argument2}
			\diag(M_{m_1}, \ldots, M_{m_p}) \subseteq \mathcal{A}_{\rho'} \subseteq \mathcal{A}_{m_1,\ldots,m_p},
		\end{equation}
		where $\mathcal{A}_{m_1,\ldots,m_p}$ is the block upper-triangular algebra (as defined in \eqref{eq:parabolicalg}).
		Let $1 \le k,l \le p$ be arbitrary and let $$(i,j) \in (m_1+\cdots+m_{k-1}, m_1+\cdots+m_{k}] \times (m_1+\cdots+m_{l-1}, m_1+\cdots+m_{l}]$$
		be arbitrary. We have
		$$(i,j) \in \rho' \iff (\underbrace{\pi^{-1}(i)}_{\in [r_k]_{\overline{\rho}}}, \underbrace{\pi^{-1}(j)}_{\in [r_l]_{\overline{\rho}}}) \in \rho \iff [r_k]_{\overline{\rho}} \preceq [r_l]_{\overline{\rho}}.$$
		The latter depends only on $k$ and $l$, implies $k \le l$, and is certainly true when $k=l$. This establishes both inclusions in \eqref{eq:Akkurt permutation argument2}.
		
	\end{proof}
	
	\begin{remark}\label{re:central decomposition}
		Let $\rho$ be a quasi-order on $\{1,\ldots,n\}$. On the same set $\{1,\ldots,n\}$, we define a new relation $\mathop{\tripprox_0} := \rho \cup \rho^t$. Explicitly:
		$$i \tripprox_0 j \iff ((i,j) \in \rho \,\text{ or }\, (j,i) \in \rho).$$
		Let $\tripprox$ be the transitive closure of $\tripprox_0$. Then $\tripprox$ is an equivalence relation on $\{1,\ldots,n\}$ and one easily checks that the centre of the SMA $\ca{A}_\rho$ is given by
		\begin{align*}
			Z(\ca{A}_\rho) &= \{\diag(\lambda_1,\ldots,\lambda_n) \in \ca{D}_n : \lambda_i = \lambda_j \text{ for all }(i,j) \in \rho^\times\}\\
			&= \{\diag(\lambda_1,\ldots,\lambda_n) \in \ca{D}_n : (\forall 1 \le i,j \le n)(i \tripprox j \implies \lambda_i = \lambda_j)\}.
		\end{align*}
		Indeed, let $D \in Z(\ca{A}_\rho)$. In particular $D \leftrightarrow \ca{D}_n$, which implies $D \in \ca{D}_n$,  If we denote  $D = \diag(\lambda_1,\ldots,\lambda_n)$, we have
		\begin{align*}
			i \tripprox_0 j &\implies (i,j) \in \rho \text{ or } (j,i) \in \rho \implies D \leftrightarrow E_{ij} \text{ or } D \leftrightarrow E_{ji}\\
			&\implies \lambda_i = \lambda_j. 
		\end{align*}
		By transitivity of $\tripprox$ we conclude  $i \tripprox j \implies \lambda_i = \lambda_j$  as well.
		
		Conversely, suppose that $D = \diag(\lambda_1,\ldots,\lambda_n) \in \ca{D}_n$ is constant on equivalence classes of $\tripprox$. We claim that $D \leftrightarrow \ca{A}_\rho$. Clearly $D \leftrightarrow \ca{D}_n$. Let $(i,j) \in \rho^\times$ be arbitrary. In particular we have $i \tripprox_0 j$ so $\lambda_i = \lambda_j$, which in turn implies $D \leftrightarrow E_{ij}$. We conclude that $D \leftrightarrow \spn\{E_{ij} : (i,j) \in \rho\}=\mathcal{A}_\rho$ which finishes the proof.
		
		\smallskip
		
		Denote by $\ca{Q}$ (or $\ca{Q}_\rho$) the quotient set of the equivalence relation $\tripprox$, i.e.
		\begin{equation}\label{eq:Qkvocx}
			\ca{Q} := \{1,\ldots,n\}/\tripprox.
		\end{equation}
		Clearly, we have $\abs{\ca{Q}} = \dim Z(\ca{A}_\rho)$. For any equivalence class $C \in \ca{Q}$ we define an idempotent
		$$P_C := \sum_{i \in C} E_{ii} \in \ca{D}_n.$$
		Obviously, $\Tr P_C = \abs{C}$ and $P_C \in Z(\ca{A}_\rho)$. In fact $(P_C)_{C \in \ca{Q}}$ is a basis for $Z(\ca{A}_\rho)$. Furthermore, by definition of the quotient set, it follows that $(P_C)_{C \in \ca{Q}}$ is a mutually orthogonal family and $\sum_{C \in \ca{Q}} P_C = I$.
		
		\smallskip
		
		Next, for each $C \in \ca{Q}$ denote by $\pi_C : C \to \{1,\ldots,\abs{C}\}$ the unique strictly increasing bijection and consider the quasi-order $$\rho_C := \{(\pi_C(i),\pi_C(j)) : (i,j) \in (C\times C) \cap \rho\}$$
		on $\{1,\ldots,\abs{C}\}$.
		Then $\ca{A}_{{\rho}_C} \subseteq M_{\abs{C}}$ is an SMA which is easily shown to be obtained from $\ca{A}_{\rho}$ by deleting all rows and columns not in $C$. Therefore, $\ca{A}_{{\rho}_C} \cong P_C \ca{A}_\rho$. Furthermore, each $\ca{A}_{{\rho}_C}$ is a central SMA (i.e. $Z(\ca{A}_{{\rho}_C})$ consists only of the scalar multiples of the identity) and there exists an algebra isomorphism
		$$\ca{A}_\rho \cong \bigoplus_{C \in \ca{Q}} \ca{A}_{{\rho}_C}.$$
		We refer to this fact as the \emph{central decomposition} of $\ca{A}_{\rho}$.
	\end{remark}
	
	\subsection{Intrinsic diagonalization}
	
	We now state our first preparatory result, regarding the intrinsic diagonalization of matrices within SMAs. Besides showcasing an interesting property of these algebras, it will be applied in an essential way in the description of Jordan embeddings between SMAs, as well as their Jordan automorphisms (Corollaries \ref{cor:SMA Jordan embeddings} and \ref{cor:JAUT-SMA}, respectively).

	\begin{theorem}\label{thm:inner diagonalization on SMA}
		Let $\ca{A}=\ca{A}_\rho \subseteq M_n$ be an SMA and let $\ca{F} \subseteq \ca{A}$ be a commuting family of diagonalizable matrices. Then there exists $S \in \ca{A}^\times$ such that $S^{-1}\ca{F}S \subseteq \ca{D}_n $.
	\end{theorem}
	
	
	Before proving Theorem \ref{thm:inner diagonalization on SMA} we consider a special case when $\ca{A}$ is contained in $\ca{T}_n$ and the family $\ca{F}$ consists of mutually orthogonal idempotents.  
	
	\begin{lemma}\label{le:simultaneous diagonalization of idempotents}
		Suppose that $\ca{P} \subseteq \ca{T}_n$ is a family of mutually orthogonal nonzero idempotents such that $\sum_{P \in \ca{P}} P = I$. Define a matrix $T \in \ca{T}_n^\times$ with $T_{ij}:= P_{ij}$ $(1 \le i \le j \le n)$, where $P \in \ca{P}$ is the unique idempotent such that $P_{jj}=1$.  
		\begin{itemize}
			\item[(a)] We have $T^{-1}\ca{P}T \subseteq \ca{D}_n$.  
			\item[(b)] If in addition $\ca{P} \subseteq \ca{A}_{\rho}$ for some SMA $\ca{A}_{\rho} \subseteq \ca{T}_n$, then $T \in \ca{A}_{\rho}^\times$.
		\end{itemize}
	\end{lemma}
	\begin{proof}
		\phantom{a}
		\begin{itemize}
			\item[(a)] Fix $P \in \ca{P}$ and let $D_P\in \ca{D}_n$ be the diagonal of $P$. Fix $1 \le i \le j \le n$ and let $Q \in \ca{P}$ be the unique idempotent such that $Q_{jj}=1$. We have
			$$(TD_P)_{ij} = \sum_{i \le k \le j} T_{ik}(D_P)_{kj} = T_{ij}P_{jj} = \begin{cases}
				P_{ij}, &\quad\text{ if } P_{jj} = 1,\\
				0, &\quad\text{ otherwise}.
			\end{cases}$$
			On the other hand, we have
			\begin{align*}
				(PT)_{ij} &= \sum_{i \le k \le j} P_{ik}T_{kj} = \sum_{i \le k < j} P_{ik}T_{kj} + P_{ij}T_{jj} = \sum_{i \le k < j} P_{ik}Q_{kj} + P_{ij}\\
				&= (PQ)_{ij} - P_{ij}Q_{jj} + P_{ij} = (PQ)_{ij} = \begin{cases}
					P_{ij}, &\quad\text{ if } P_{jj} = 1,\\
					0, &\quad\text{ otherwise}.
				\end{cases}
			\end{align*}
			This closes the proof of $TD_P = PT$. 
			\item[(b)] The statement follows directly from
			$$\supp T \subseteq \bigcup_{P \in \ca{P}} \supp P \subseteq \rho.$$
		\end{itemize}
	\end{proof}

	\begin{remark}\label{re:spectral projections}
		Let $A \in M_n$ be a diagonalizable matrix. By an elementary linear algebra argument, there exist a unique family $\{P_\lambda : \lambda \in \sigma(A)\}$ of mutually orthogonal nonzero idempotents (spectral idempotents of $A$) such that
		$$A = \sum_{\lambda \in \sigma(A)} \lambda P_\lambda, \qquad \sum_{\lambda \in \sigma(A)} P_\lambda = I.$$
		For each $\lambda \in \sigma(A)$ let $p_\lambda \in \C[x]$ be the unique polynomial of degree $< \abs{\sigma(A)}$ such that
		$$p_\lambda(x)= \begin{cases}
			1, &\quad\text{ if }x = \lambda,\\
			0, &\quad\text{ if }x \in \sigma(A)\setminus\{\lambda\}.
		\end{cases}$$
		
		Then it is well-known (and easy to verify) that $p_\lambda(A) = P_\lambda$ for all $\lambda \in \sigma(A)$.  In particular, all spectral projections of a diagonalizable matrix $A\in M_n$ are contained in the unital subalgebra $\C[A]$ of $M_n$ (generated by $A$).  Therefore, any unital (Jordan) subalgebra $\ca{A}$ of $M_n$  containing $A$ also contains all spectral projections of $A$.
	\end{remark}
	
	\begin{proof}[Proof of Theorem \ref{thm:inner diagonalization on SMA}]
		
		By an elementary argument, it suffices to prove the statement when $\ca{F}$ is finite (otherwise we conduct the argument on the basis of the linear span of $\ca{F}$, which is again a finite commuting subset of $\ca{A}$ consisting of diagonalizable matrices). For simplicity, we can further assume that $|\ca{F}|=2$, so that $\ca{F} = \{A,B\}$, as the general case follows analogously.
		
		\begin{case}\label{cl:inner diagonalization SMA contained in Tn}
			We prove the statement when $\ca{A} \subseteq \ca{T}_n$.
		\end{case}
		Let $\{P_\lambda : \lambda \in \sigma(A)\}$ and $\{Q_\mu : \mu \in \sigma(B)\}$ be the families of spectral projections of $A$ and $B$, respectively. By Remark \ref{re:spectral projections}, both families are in fact contained in $\ca{A}$.
		
		\smallskip
		
		Since $A \leftrightarrow B$, we have $P_\lambda \leftrightarrow Q_\mu$ for all $\lambda\in \sigma(A)$ and $\mu \in \sigma(B)$ (also by Remark \ref{re:spectral projections}). Using this, it is easy to show that $$\ca{P} := \{P_\lambda Q_\mu : \lambda\in \sigma(A), \mu \in \sigma(B)\}\setminus \{0\} \subseteq \ca{A}$$
		is a family of mutually orthogonal nonzero idempotents such that $\sum_{R \in \ca{P}} R = I$. Also, both $A$ and $B$ are linear combinations of elements of $\ca{P}$. We can apply Lemma \ref{le:simultaneous diagonalization of idempotents} to $\ca{P}$ to obtain $T \in \ca{A}^\times$ which diagonalizes the entire family $\ca{P}$, and consequently $A$ and $B$ as well.
		
		\begin{case}\label{cl:inner diagonalization SMA contained in parabolic subalgebra}
			We prove the statement when $$\diag(M_{k_1}, \ldots, M_{k_p}) \subseteq \ca{A} \subseteq \ca{A}_{k_1, \ldots, k_p}$$ (where $\ca{A}_{k_1, \ldots, k_p}$ is the corresponding block upper-triangular algebra, as defined by \eqref{eq:parabolicalg}).
		\end{case}
		Denote by  $
		X_1,Y_1 \in M_{k_1},\ldots, X_p,Y_p \in M_{k_p}$ the diagonal blocks of $A$ and $B$, respectively. Since $A \leftrightarrow B$, for each $1 \le j \le p$ we have $X_j \leftrightarrow Y_j$ so by the Schur triangularization (see e.g.\ \cite[Theorem 40.5]{Prasolov}) we can choose $U_j \in M_{k_j}^\times$ such that $U_j X_j U_j^{-1},  U_j Y_j U_j^{-1}\in \ca{T}_{k_j}$. Then $$U := \diag(U_1,\ldots, U_p) \in \diag(M_{k_1}, \ldots, M_{k_p}) \subseteq \ca{A}$$
		and 
		$$U AU^{-1}, U BU^{-1} \in \ca{A} \cap \ca{T}_n.$$
		It follows that $\{UAU^{-1}, UBU^{-1}\}$ is a commuting family of diagonalizable matrices, so we can apply Case \ref{cl:inner diagonalization SMA contained in Tn} on the SMA $\ca{A} \cap \ca{T}_n$ to obtain $S \in (\ca{A} \cap \ca{T}_n)^\times$ such that
		$$UAU^{-1}, UBU^{-1} \in S\ca{D}_n S^{-1}$$
		or equivalently $A,B \in R\ca{D}_n R^{-1}$, where $R:=U^{-1}S \in \ca{A}^\times$. 
		
		\begin{case}
			Now assume that $\ca{A}$ is a general SMA.
		\end{case}
		
		From Lemma \ref{le:Akkurt permutation argument} we know that there exists a permutation matrix $R \in M_n^\times$ and a block upper-triangular subalgebra $\ca{A}_{k_1, \ldots, k_p} \subseteq M_n$ such that the SMA $R \ca{A} R^{-1}$ satisfies $$\diag(M_{k_1}, \ldots, M_{k_p}) \subseteq R \ca{A} R^{-1} \subseteq \ca{A}_{k_1, \ldots, k_p}.$$
		
		Since $A,B \in \ca{A}$, we have $R A R^{-1},R B R^{-1} \in R \ca{A} R^{-1}$ so by Case \ref{cl:inner diagonalization SMA contained in parabolic subalgebra}, there exists $S \in (R \ca{A} R^{-1})^\times = R \ca{A}^\times R^{-1}$ such that $$R A R^{-1},R B R^{-1} \in S\ca{D}_nS^{-1}.$$
		It follows that $$A,B \in (R^{-1}SR) R^{-1}\ca{D}_n R(R^{-1}SR)^{-1},$$
		where $R^{-1}SR \in \ca{A}^\times$. The proof is now complete.
	\end{proof}
	
	\begin{remark}
		Not all unital subalgebras $\ca{A} \subseteq M_n$ possess the intrinsic diagonalization property in the sense of Theorem \ref{thm:inner diagonalization on SMA}. For instance, fix an arbitrary diagonalizable matrix $T \in M_n\setminus \ca{D}_n$ and let $\ca{A} :=\C[T] \subseteq M_n$. Then $\ca{A}$ is a unital commutative subalgebra of $M_n$ (its dimension is equal to the degree of the minimal polynomial of $T$), so trivially a matrix $A \in \ca{A}$ is of the form $A = SDS^{-1}$ for some $S \in \ca{A}^\times$ and $D \in \ca{D}_n$  if and only if $A$ itself is diagonal. However, this is clearly not true for the matrix $T \in \ca{A}$.
		
		\smallskip
		
		Furthermore, let
		$$T:= \begin{bmatrix}
			0 & 0 & 1 \\ 0 & 1 & 0 \\ 0 & 0 & 2 
		\end{bmatrix} \in \ca{T}_3$$
		and
		$$\ca{A} :=\C[T]=\left\{\begin{bmatrix}
			x & 0 & \frac12(z-x) \\ 0 & y & 0 \\ 0 & 0 & z 
		\end{bmatrix} : x,y,z \in \C\right\} \subseteq \ca{T}_3.$$
		One easily verifies that the algebra $\ca{A}$ is conjugate to $\ca{D}_3$ (e.g.\ $\ca{A} = (T+E_{11})\ca{D}_3 (T+E_{11})^{-1}$). In particular, the property of being conjugated to an SMA is not  sufficient to ensure the intrinsic diagonalization in the sense of Theorem \ref{thm:inner diagonalization on SMA}.
	\end{remark}
	
	\section{Jordan embeddings of structural matrix algebras}\label{sec:MAMP}
	
	\subsection{Algebra embeddings}
	Let $\rho$ be a  quasi-order on $\{1,\ldots,n\}$. Following \cite{Coelho}, we say that a map $g : \rho \to \C^\times$ is \emph{transitive} if 
	$$g(i,j)g(j,k) = g(i,k),  \qquad \text{ for all } (i,j), (j,k) \in \rho.$$
	Note that necessarily $g|_{\Delta_n} \equiv 1$, so it suffices to verify the above condition for $(i,j), (j,k) \in \rho^\times$. We say that a transitive map $g : \rho \to \C^\times$ is \emph{trivial} if there exists a map $s : \{1,\ldots,n\} \to \C^\times$ such that $g$ \emph{separates through $s$}, that is
	$$g(i,j) = \frac{s(i)}{s(j)}, \qquad \text{ for all }(i,j) \in \rho$$
	(again, it suffices to stipulate this for all $(i,j) \in \rho^\times$). Every transitive map $g$ induces an (algebra) automorphism $g^* : \ca{A}_\rho \to \ca{A}_\rho$ given on the basis of  matrix units as
	$$g^*(E_{ij}) = g(i,j)E_{ij}, \qquad \text{ for all }(i,j) \in\rho.$$
	Triviality of a transitive map reflects on the induced automorphism in a natural way:
	\begin{lemma}{\cite[Lemma 4.10]{Coelho}}\label{le:g trivial iff induced map is inner automorphism}
		Let $g : \rho \to \C^\times$ be a transitive map. Then $g$ is trivial if and only if there exists $T \in M_n^\times$ such that the induced automorphism $g^* : \ca{A}_\rho \to \ca{A}_\rho$ is of the form $g^*(\cdot) = T(\cdot) T^{-1}$.  In this case, $T$ is in fact a diagonal matrix, so that  $g^*$ is an inner automorphism of $\ca{A}_\rho$.
	\end{lemma}
	
	We now state the aforementioned description of all automorphisms of SMAs, which was first obtained by Coelho in \cite[Theorem C]{Coelho} and later streamlined by Akkurt et al. as:
	
	\begin{theorem}{\cite[Theorem 2.2 (Factorization Theorem)]{Akkurt}}\label{thm:description of Aut(A_rho)}
		Let $\ca{A}_\rho\subseteq M_n$ be an SMA. A map $\phi : \ca{A}_\rho \to \ca{A}_\rho$ is an algebra automorphism if and only if there exists an invertible matrix $T \in \ca{A}_\rho^\times$, a transitive map $g : \rho \to \C^\times$ and a $(\rho, \rho)$-increasing permutation $\pi \in S_n$ such that
		$$\phi(\cdot) = (TR_{\pi}) g^*(\cdot)(TR_{\pi})^{-1}.$$
	\end{theorem}
	\begin{remark}
		In Theorem \ref{thm:description of Aut(A_rho)}, the three parameters $T, g, \pi$  which represent an automorphism $\phi : \ca{A}_\rho \to \ca{A}_\rho$ are in general not unique. In fact, there is a significant degree of redundancy, which is clear from Coelho's group theoretic formulation of this result (\cite[Theorem C]{Coelho}). To illustrate the issue, consider the quasi-order 
		$$\rho := \Delta_3 \cup \{(1,2),(1,3),(2,3),(3,2)\}$$
		and the corresponding SMA 
		$$\ca{A}_\rho = \begin{bmatrix}
			* & * & * \\
			0 & * & * \\
			0 & * & *
		\end{bmatrix}.$$
		Let $\phi= \id : \ca{A}_\rho \to \ca{A}_\rho$ be the trivial automorphism. We have 
		$$\phi(\cdot) = (T_1 R_{\pi_1}) g_1^*(\cdot) (T_1 R_{\pi_1})^{-1} = (T_2 R_{\pi_2}) g_2^*(\cdot) (T_2 R_{\pi_2})^{-1},$$
		where
		\begin{itemize}
			\item $\pi_1$ is the identity, $T_1 = I$ and $g_1 : \rho \to \C^\times$ is the constant map $1$.
			\item $\pi_2$ is the transposition $2 \leftrightarrow 3$,  $T_2 = \begin{bmatrix}
				\frac12 & 0 & 0 \\
				0 & 0 & 1 \\
				0 & 1 & 0
			\end{bmatrix}$ and $g_2 : \rho \to \C^\times$ is given by
			$$g_2(i,j) = \begin{cases}
				2, &\text{ if }(i,j) \in \{(1,2),(1,3)\},\\
				1, &\text{ if }(i,j) \in \rho\setminus\{(1,2),(1,3)\}.
			\end{cases}$$
		\end{itemize}    
	\end{remark}
	
	We explicitly state the following simple generalization of Theorem \ref{thm:description of Aut(A_rho)}, as it motivates our later result regarding Jordan embeddings between SMAs (Corollary \ref{cor:SMA Jordan embeddings}).
	\setcounter{case}{0}
	\begin{corollary}\label{cor:SMA embedding}
		Let $\ca{A}_\rho, \ca{A}_{\rho'} \subseteq M_n$ be SMAs. Then $\ca{A}_{\rho}$ embeds (as an algebra) into $\ca{A}_{\rho'}$ if and only if there exists a $(\rho,\rho')$-increasing permutation $\pi \in S_n$.
		Furthermore, if $\phi : \ca{A}_\rho \to \ca{A}_{\rho'}$ is an algebra embedding, then there exists an invertible matrix $T \in \ca{A}_{\rho'}^\times$, a $(\rho,\rho')$-increasing permutation $\pi \in S_n$ and a transitive map $g : \rho \to \C^\times$ such that
		$$\phi(\cdot) = (T R_{\pi})  g^*(\cdot) (T R_{\pi})^{-1}.$$
	\end{corollary}
	\begin{proof}
		If there exists a $(\rho,\rho')$-increasing permutation $\pi \in S_n$ then, following the discussion around \eqref{eq:conjugation by permutation}, the map  $R_\pi(\cdot)R_{\pi}^{-1}$ defines the algebra embedding $\ca{A}_\rho \to \ca{A}_{\rho'}$. 
		
		\smallskip
		
		Conversely, assume that $\phi : \ca{A}_\rho \to \ca{A}_{\rho'}$ is an algebra embedding. Note that $$\{\phi(E_{11}), \ldots, \phi(E_{nn})\} \subseteq \ca{A}_{\rho'}$$ is a set of mutually orthogonal idempotents so by Theorem \ref{thm:inner diagonalization on SMA} (in fact, Lemma \ref{le:simultaneous diagonalization of idempotents} suffices) there exists $S \in \ca{A}_{\rho'}^\times$ and a permutation $\pi \in S_n$ such that $$\phi(E_{ii}) = SE_{\pi(i)\pi(i)}S^{-1} = (SR_\pi)E_{ii}(SR_{\pi})^{-1}, \qquad \text{ for all } 1 \le i \le n.$$
		For each $(i,j) \in \rho^\times$ we have
		$$\phi(E_{ij}) = \phi(E_{ii}E_{ij}E_{jj}) = ((SR_\pi)E_{ii}(SR_{\pi})^{-1})\phi(E_{ij})((SR_\pi)E_{jj}(SR_{\pi})^{-1})$$
		and hence $\phi(E_{ij}) = (SR_\pi)(g(i,j)E_{ij})(SR_{\pi})^{-1}$ for some nonzero scalar $g(i,j) \in \C^\times$. Multiplicativity of $\phi$ directly implies that the map
		$$
		g : \rho \to \C^\times, \qquad (i,j)  \mapsto \begin{cases} g(i,j), &\quad\text{ if } (i,j) \in \rho^\times, \\
			1, &\quad\text{ if }  i=j
		\end{cases}    
		$$
		is transitive. We conclude $\phi = (SR_\pi)g^*(\cdot)(SR_{\pi})^{-1}$. For all $X \in \mathcal{A}_\rho$ we have
		$$R_{\pi} g^*(X) R_{\pi}^{-1} = S^{-1}\phi(X)S \in \mathcal{A}_{\rho'}.$$
		By the surjectivity of $g^*$, it follows that $R_{\pi} \ca{A}_\rho R_{\pi}^{-1} \subseteq \ca{A}_{\rho'}$ so $\pi$ is $(\rho,\rho')$-increasing.
	\end{proof}
	
	To end this subsection, we also mention Coelho's characterization of SMAs which admit only inner automorphisms.
	\begin{theorem}{\cite[Theorem D]{Coelho}}\label{thm:when are all automorphism inner?}
		Let $\ca{A}_\rho \subseteq M_n$ be an SMA. Then every automorphism $\ca{A}_\rho \to \ca{A}_\rho$ is inner if and only if both of the following hold:
		\begin{enumerate}[(a)]
			\item every transitive map $g : \rho \to \C^\times$ is trivial,
			\item every quasi-order automorphism of $\rho$ fixes the equivalence classes of $\overline{\rho}$.
		\end{enumerate}
	\end{theorem}
	
	\subsection{Jordan embeddings}
	Now we are ready to describe the general form of Jordan embeddings $\phi:\mathcal{A}\to \mathcal{B}$ between SMAs $\ca{A}$ and $\ca{B}$ of $M_n$.
	As expected, the transitive maps will play a role of similar importance in the description of all Jordan embeddings and rank(-one) preservers $\ca{A} \to M_n$. Their appearance displays the relative complexity of (Jordan) algebraic and preserver theory on SMAs when compared to $M_n$, $\ca{T}_n$ or the block upper-triangular subalgebras (see \cite{GogicPetekTomasevic}). On the other hand, permutation matrices appear only when we restrict the codomain to $\ca{B}$, i.e.\ in the description of Jordan embeddings $\ca{A} \to \ca{B}$.
	
	\smallskip
	
	We first start with the special case when $\mathcal{B}=M_n$.
	
	\begin{lemma}\label{le:preserves diagonalizable}
		Let $\ca{A}_\rho \subseteq M_n$ be an SMA and let $\phi : \ca{A}_\rho \to M_n$ be a Jordan homomorphism such that $\phi(E_{ij}) \ne 0$ for all $(i,j)\in \rho$. Then there exists an invertible matrix $S\in M_n^\times$ such that $\phi(D) = SDS^{-1}$ for all $D \in \ca{D}_n$.
	\end{lemma}
	\begin{proof}
		By Lemma \ref{le:Jordan homomorphism basic properties} (c) we conclude that $\phi(E_{11}),\ldots, \phi(E_{nn})$ is a family of mutually orthogonal nonzero idempotents. Therefore, there exists $S \in M_n^\times$ such that $$\phi(E_{kk}) = SE_{kk}S^{-1}, \qquad 1 \le k \le n.$$
		The claim follows by linearity.
	\end{proof}
	
	\begin{lemma}\label{le:Eij or Eji}
		Let $\ca{A}_\rho \subseteq M_n$ be an SMA and let $\phi : \ca{A}_\rho \to M_n$ be a Jordan homomorphism such that  $\phi(E_{ij}) \ne 0$ for all $(i,j)\in \rho$ and $\phi|_{\ca{D}_n}$ is the identity. Then for every $(i,j) \in \rho^\times$ there exist scalars $\alpha_{ij}, \beta_{ij} \in \C$, exactly one of which is zero, such that $\phi(E_{ij}) = \alpha_{ij}E_{ij} + \beta_{ij}E_{ji}.$
	\end{lemma}
	\begin{proof}
		Fix $(i,j) \in \rho^\times$. By Lemma \ref{le:Jordan homomorphism basic properties} (b) we have
		$$\phi(E_{ij}) = \phi(E_{ii}E_{ij}E_{jj} + E_{jj}E_{ij}E_{ii}) = E_{ii}\phi(E_{ij})E_{jj} + E_{jj}\phi(E_{ij})E_{ii}.$$
		Therefore, $\phi(E_{ij})$ is supported in $\{(i,j),(j,i)\}$ so there exist scalars $\alpha_{ij}, \beta_{ij} \in \C$ such that
		$$\phi(E_{ij}) = \alpha_{ij}E_{ij} + \beta_{ij}E_{ji}.$$
		Furthermore, we have $$0 = \phi(E_{ij}^2) = \phi(E_{ij})^2 = \alpha_{ij}\beta_{ij}(E_{ii}+E_{jj})$$ 
		so exactly one of $\alpha_{ij}$ and $\beta_{ij}$ is equal to zero $0$ (as none of the matrix units  are in the kernel of $\phi$, by the assumption).
	\end{proof}
	
	\begin{lemma}\label{le:MAMP}
		Let $\ca{A}_\rho \subseteq M_n$ be an SMA and let $\phi : \ca{A}_\rho \to M_n$ be a Jordan homomorphism such that  $\phi(E_{ij}) \ne 0$ for all $(i,j)\in \rho$ and $\phi|_{\ca{D}_n}$ is the identity. Define
		$$\rho^\phi_M := \{(i,j) \in \rho : \phi(E_{ij}) \parallel E_{ij}\}, \qquad \rho^\phi_A := \{(i,j) \in \rho : \phi(E_{ij}) \parallel E_{ji}\}.$$
		We have 
		\begin{enumerate}[(a)]
			\item $\rho^\phi_M \cup \rho^\phi_A = \rho$ and $\rho^\phi_M \cap \rho^\phi_A = \Delta_n$.
			\item $\rho^\phi_M$ and $\rho^\phi_A$ are quasi-orders on $\{1,\ldots,n\}$ and there exist transitive maps $g : \rho^\phi_M \to \C^\times$ and $h : \rho^\phi_A \to \C^\times$ such that the restrictions
			$$\phi|_{\ca{A}_{\rho^\phi_M}} : \ca{A}_{\rho^\phi_M} \to M_n, \qquad \phi|_{\ca{A}_{\rho^\phi_A}} : \ca{A}_{\rho^\phi_A} \to M_n$$
			are equal to $g^*(\cdot)$ and $h^*(\cdot)^t$, respectively. In particular, the maps $\phi|_{\ca{A}_{\rho^\phi_M}}$ and $\phi|_{\ca{A}_{\rho^\phi_A}}$ are multiplicative and antimultiplicative, respectively.
			\item Suppose $(i,j) \in \rho^\times$. Then $$
			(\{i\} \times \rho(i)) \cup (\rho^{-1}(i) \times \{i\}) \cup (\{j\} \times \rho(j)) \cup (\rho^{-1}(j) \times \{j\}) \subseteq \rho^{\phi}_M$$
			or
			$$
			(\{i\} \times \rho(i)) \cup (\rho^{-1}(i) \times \{i\}) \cup (\{j\} \times \rho(j)) \cup (\rho^{-1}(j) \times \{j\}) \subseteq \rho^{\phi}_A.$$
			\item Let $P \in \ca{D}_n$ be a diagonal idempotent defined by
			$$P_{ii} = 1 \iff \text{there exists $1 \le j \le n, j \ne i$ such that $(i,j) \in \rho^\phi_M$ or $(j,i) \in \rho^\phi_M$}.$$
			Then $P\in Z(\ca{A}_\rho)$, $PX \in \ca{A}_{\rho^\phi_M}$ and $(I-P)X \in \ca{A}_{\rho^\phi_A}$ for all $X \in \ca{A}_\rho$.
			
			\item Suppose that $(i,j),(j,k) \in \rho^\times$. Then either $(i,j),(j,k) \in \rho^\phi_M$ or $(i,j),(j,k) \in \rho^\phi_A$.
			
		\end{enumerate}
	\end{lemma}
	\begin{proof}
		Following the notation from Lemma \ref{le:Eij or Eji}, throughout the proof for each $(i,j) \in \rho^\phi_M$ by $\alpha_{ij} \in \C^\times$ we denote the unique nonzero scalar such that $\phi(E_{ij}) = \alpha_{ij}E_{ij}$, while for each $(i,j) \in \rho_A$ by $\beta_{ij} \in \C^\times$ we denote the unique nonzero scalar such that $\phi(E_{ij}) = \beta_{ij}E_{ji}$.
		
		\begin{enumerate}[(a)]
			\item Clearly, for all $1 \le i \le n$ we have $\phi(E_{ii}) = E_{ii}$ so $(i,i)$ is contained in both $\rho^\phi_M $ and $\rho^\phi_A$. It follows that $\Delta_n \subseteq \rho^\phi_M \cap \rho^\phi_A $. On the other hand, let $(i,j) \in \rho^\times$ be arbitrary. Lemma \ref{le:Eij or Eji} directly implies that either $(i,j) \in \rho^\phi_M$ or $(i,j) \in \rho^\phi_A$. We conclude $\rho^\phi_M \cup \rho^\phi_A = \rho$ and $\rho^\phi_M \cap \rho^\phi_A \subseteq \Delta_n$.

			\item
			
			We first prove that $\rho^\phi_M$ is a quasi-order and that there exists a transitive map $g : \rho^\phi_M \to \C^\times$ such that $\phi|_{\ca{A}_{\rho^\phi_M}} = g^*(\cdot)$. The reflexivity of $\rho^\phi_M$ follows from (a). Suppose that $(i,j),(j,k) \in \rho^\phi_M$. We show $(i,k) \in \rho^\phi_M$ and $\alpha_{ik} = \alpha_{ij}\alpha_{jk}$. Both of these are clear if $i=j$ or $j=k$ so we can further assume $(i,j),(j,k) \in (\rho^\phi_M)^\times$. We have
			\begin{align*}
				\phi(E_{ik}) + \delta_{ki} E_{jj} &= \phi(E_{ik} + \delta_{ki}E_{jj}) \\
				&= \phi(E_{ij}E_{jk} + E_{jk}E_{ij}) = \phi(E_{ij})\phi(E_{jk}) + \phi(E_{jk})\phi(E_{ij}) \\
				&= \alpha_{ij}\alpha_{jk}(E_{ik} + \delta_{ki}E_{jj}),
			\end{align*}
			where, as usual, $\delta_{ij}$ denotes the Kronecker delta symbol. If $i = k$, then $(i,i) \in \rho^\phi_M$ is trivial and the above relation reduces to
			$$E_{ii} + E_{jj} = \alpha_{ij}\alpha_{ji}(E_{ii} + E_{jj})$$
			which implies $\alpha_{ij}\alpha_{ji} = 1$. On the other hand, if $i \ne k$, then the above relation reduces to
			$$\phi(E_{ik}) = \alpha_{ij}\alpha_{jk}E_{ik}$$
			which first implies $(i,k) \in \rho^{\phi}_M$ and then $\alpha_{ik} = \alpha_{ij}\alpha_{jk}$. It follows directly that the map $g : \rho^{\phi}_M \to \C^\times, g(i,j) := \alpha_{ij}$ is transitive and $\phi|_{\ca{A}_{\rho^\phi_M}} = g^*(\cdot)$.
			
			To show the second claim, consider the map
			$\phi^t : \ca{A}_{\rho} \to M_n$, given by $X \mapsto \phi(X)^t$. Obviously, $\phi^t$ satisfies the same properties as $\phi$ and $\rho^\phi_A = \rho^{\phi^t}_M.$
			By the first part of the proof it follows that $\rho^\phi_A$ is a quasi-order and that there exists a transitive map $h : \rho^{\phi^t}_M \to \C^\times$ such that $\phi^t|_{\ca{A}_{\rho^{\phi^t}_M}} = h^*(\cdot)$. Then clearly $\phi|_{\ca{A}_{\rho^{\phi}_A}} = h^*(\cdot)^t$.
			
			\item For concreteness assume that $(i,j) \in \rho^{\phi}_M$. Let $(j,k) \in \rho^\times$ be arbitrary. We claim that $(j,k) \in \rho^{\phi}_M$. Assume the contrary, that $(j,k) \in \rho^{\phi}_A$. Then we have
			\begin{align*}
				\phi(E_{ik}) + \delta_{ik}E_{jj}&= \phi(E_{ik} + \delta_{ik}E_{jj}) \\
				&= \phi(E_{ij}E_{jk} + E_{jk}E_{ij}) \\
				&= \phi(E_{ij})\phi(E_{jk}) + \phi(E_{jk})\phi(E_{ij}) \\
				&= \alpha_{ij}\beta_{jk}(E_{ij}E_{kj} + E_{kj}E_{ij}) \\
				&= 0,
			\end{align*}
			which is a contradiction. Therefore, $(j,k) \in \rho^{\phi}_M$, and consequently $(i,k) \in \rho^{\phi}_M$ by (b). It follows that $(\{i\} \times \rho(i)) \cup (\{j\} \times \rho(j)) \subseteq \rho^{\phi}_M$. The proof of $(\rho^{-1}(i) \times \{i\}) \cup (\rho^{-1}(j) \times \{j\}) \subseteq \rho^{\phi}_M$  as well as of the corresponding statement for $\rho^{\phi}_A$ is analogous.
			\item It suffices to show that for all $(i,j) \in \rho$ we have $P \leftrightarrow E_{ij}$, $PE_{ij} \in \ca{A}_{\rho^\phi_M}$ and $(I-P)E_{ij} \in \ca{A}_{\rho^\phi_A}$. Since all three claims are trivially true when $i=j$, fix $(i,j) \in \rho^\times$. By (a), we consider two cases:
			\begin{itemize}
				\item If $(i,j) \in \rho^\phi_M$, then $P_{ii} = P_{jj} = 1$ by definition, so
				$$PE_{ij} = \underbrace{E_{ij}}_{\in \ca{A}_{\rho^\phi_M}} = E_{ij}P \implies (I-P)E_{ij} = 0 \in \ca{A}_{\rho^\phi_A},$$
				which establishes all three claims.
				
				\item If $(i,j) \in \rho^\phi_A$, then $P_{ii} = P_{jj} = 0$ as a consequence of (c), so
				$$PE_{ij} = \underbrace{0}_{\in \ca{A}_{\rho^\phi_M}} = E_{ij}P \implies (I-P)E_{ij} = E_{ij} \in \ca{A}_{\rho^\phi_A},$$
				which establishes all three claims.
			\end{itemize}
			
			\item This is a direct consequence of (c).
		\end{enumerate}
	\end{proof}
	
	\begin{theorem}\label{thm:general form of phi}
		Let $\ca{A}_\rho \subseteq M_n$ be an SMA and let $\phi : \ca{A}_\rho \to M_n$ be a Jordan homomorphism such that $\phi(E_{ij}) \ne 0$ for all $(i,j)\in \rho$. Then there exists an invertible matrix $S \in M_n^\times$, a central idempotent $P \in Z(\ca{A}_\rho)$, and a transitive map $g : \rho \to \C^\times$ such that
		$$\phi(\cdot) = S(Pg^*(\cdot) + (I-P)g^*(\cdot)^t)S^{-1}.$$
		In particular, $\phi$ is injective (i.e.\ a Jordan embedding).
	\end{theorem}
	\begin{proof}
		By Lemma \ref{le:preserves diagonalizable} there exists $S \in M_n^\times$ such that $\phi(D) = SDS^{-1}$ for all $D \in \ca{D}_n$. By passing onto the map $S^{-1}\phi(\cdot)S$ which satisfies the same properties, without loss of generality we assume that $\phi|_{\ca{D}_n}$ is the identity.\smallskip
		
		Let  $g : \rho^\phi_M \to \C^\times$ and $h : \rho^\phi_A \to \C^\times$ be transitive maps from Lemma \ref{le:MAMP} (b).  Define a map
		$$f : \rho \to \C^\times, \qquad f(i,j) = \begin{cases}
			g(i,j), &\quad \text{ if } (i,j) \in \rho^{\phi}_M,\\
			h(i,j), &\quad \text{ if } (i,j) \in \rho^{\phi}_A.
		\end{cases}$$
		Lemma \ref{le:MAMP} (e) directly implies that $f$ is a transitive map. Clearly, $f^*|_{\ca{A}_{\rho^{\phi}_M}} = g^*$ and $f^*|_{\ca{A}_{\rho^{\phi}_A}} = h^*$.\smallskip
		
		Let $P \in Z(\ca{A}_\rho)$ be the central idempotent defined in Lemma \ref{le:MAMP} (d). As $PX \in \ca{A}_{\rho^{\phi}_M}$ and $(I-P)X \in \ca{A}_{\rho^{\phi}_A}$ for all for $X \in \ca{A}_\rho$, we have 
		\begin{align*}
			\phi(X) &= \phi(PX) + \phi((I-P)X) = g^*(PX) + h^*((I-P)X)^t \\
			&= f^*(PX) + f^*((I-P)X)^t = f^*(P)f^*(X) + f^*(X)^t f^*(I-P)^t \\
			&= Pf^*(X) + (I-P)f^*(X)^t.
		\end{align*}

		It remains to show the injectivity of $\phi$. Let $X \in \ca{A}_\rho$ be a matrix such that $\phi(X) = 0$. By left-multiplying the expression
		$$0 = \phi(X) = Pf^*(X) + (I-P)f^*(X)^t$$
		by $P$ and $I-P$ respectively, we conclude $Pf^*(X) = (I-P)f^*(X)^t = 0$. Since $I-P \in Z(\ca{A}_\rho) \subseteq \ca{D}_n$, the latter equality can be transposed to yield $(I-P)f^*(X) = 0$. Overall, we obtain
		$$0 = Pf^*(X) + (I-P)f^*(X) = f^*(X),$$
		which implies $X = 0$ by the injectivity of $f^*$. We conclude that $\phi$ is injective.
		
	\end{proof}

	\begin{corollary}\label{cor:MAMP equivalences}
		Let $\ca{A}_\rho \subseteq M_n$ be an SMA. The following conditions are equivalent:
		\begin{enumerate}[(i)]
			\item Every Jordan embedding $\ca{A}_\rho \to M_n$ is multiplicative or antimultiplicative.
			\item The quotient set $\ca{Q}$ defined by \eqref{eq:Qkvocx} contains at most one class $C\in \ca{Q}$ with $\abs{C} \ge 2$.
		\end{enumerate}
	\end{corollary}
	\begin{proof}
		
		\fbox{$(i) \implies (ii)$} 
		We prove the contrapositive. Suppose that there exist $C_1,C_2 \in \ca{Q}$, $C_1 \ne C_2$ such that $\abs{C_1}, \abs{C_2} \ge 2$. Let $$P := \sum_{i\in C_1} E_{ii} \in \ca{D}_n.$$
		By Remark \ref{re:central decomposition}, $P$ is a central projection and therefore, the map
		$$\phi : \ca{A}_\rho \to M_n, \qquad \phi(\cdot) = P(\cdot) + (I-P)(\cdot)^t$$
		is a Jordan embedding which is neither multiplicative nor antimultiplicative. Indeed, choose some $i,j \in C_2$ such that $(i,j) \in \rho^\times$. Then
		$$\phi(E_{ij}E_{jj}) = \phi(E_{ij}) = E_{ji} \ne 0 = E_{ji}E_{jj}=\phi(E_{ij})\phi(E_{jj})$$
		shows that $\phi$ is not multiplicative. That $\phi$ is not antimultiplicative can be shown in a similar way by choosing elements of $C_1$.
		\smallskip

		\fbox{$(ii) \implies (i)$} Suppose that there exists $C \in \ca{Q}$, $\abs{C} \ge 1$ such that $$\ca{Q} = \{C\} \cup \{\{i\} : i \in \{1,\ldots,n\}\setminus C\}.$$
		Let $\phi : \ca{A}_\rho \to M_n$ be a Jordan embedding. By Theorem \ref{thm:general form of phi} there exists an invertible matrix $S \in M_n^\times$, a central idempotent $P \in Z(\ca{A}_\rho)$, and a transitive map $g : \rho \to \C^\times$ such that
		$$\phi(\cdot) = S(Pg^*(\cdot) + (I-P)g^*(\cdot)^t)S^{-1}.$$
		Notice that for all $i \in \{1,\ldots,n\}\setminus C$ we have
		$$E_{ii}X = E_{ii}X^t, \qquad \text{ for all }X \in \ca{A}_{\rho}.$$
		Recall from Remark \ref{re:central decomposition} that both $P$ and $I-P$ are sums of $E_{ii}$ for $i \in \{1,\ldots,n\}\setminus C$ and $P_C = \sum_{j \in C}E_{jj}$. Therefore, for $C$ and each $i \in \{1,\ldots,n\}\setminus C$ there is a map $(\cdot)^\circ_C, (\cdot)^{\circ_i} \in \{\id,(\cdot)^t\}$ such that
		\begin{align*}
			\phi(\cdot) &= S\left(\sum_{i \in \{1,\ldots,n\}\setminus C} E_{ii}g^*(\cdot)^{\circ_i} + P_C g^*(\cdot)^{\circ_C}\right)S^{-1}\\
			&= S\left(\sum_{i \in \{1,\ldots,n\}\setminus C} E_{ii}g^*(\cdot)^{\circ_C} + P_C g^*(\cdot)^{\circ_C}\right)S^{-1}\\
			&= Sg^*(\cdot)^{\circ_C} S^{-1}.
		\end{align*}
		We conclude that $\phi$ is multiplicative or antimultiplicative (depending on whether $\circ_C$ is the identity map or the transposition map).
	\end{proof}
	
	Using Theorem  \ref{thm:inner diagonalization on SMA} we can now prove the following consequences of Theorem \ref{thm:general form of phi} regarding Jordan embeddings between two SMAs. Before stating them, let us introduce the following auxiliary notation. Let $\rho$ be a quasi-order on $\{1,\ldots,n\}$ and consider a subset $\ca{U} \subseteq \{1,\ldots,n\}$ expressible as a union of some equivalence classes of $\tripprox$ (i.e.\ $\ca{U}$ is a union of some classes of the quotient set $\ca{Q}_{\rho}$). Denote $\ca{U}^c := \{1,\ldots,n\} \setminus \ca{U}$ and define a relation on $\{1,\ldots,n\}$ by
	\begin{equation}\label{eq:wierd quasi order}
		\rho^{\ca{U}} := \left(\rho \cap (\ca{U}\times \ca{U}) \right) \cup \left(\rho^t \cap (\ca{U}^c \times \ca{U}^c)\right),
	\end{equation}
	which is easily seen to be a quasi-order.
	
	{\begin{corollary}\label{cor:SMA Jordan embeddings}
			
			Let $\ca{A}_\rho, \ca{A}_{\rho'} \subseteq M_n$ be SMAs. 
			Then $\ca{A}_{\rho}$ Jordan-embeds into $\ca{A}_{\rho'}$ if and only if there exists a subset \,$\ca{U} \subseteq \{1,\ldots,n\}$ expressible as a union of classes of $\ca{Q}_{\rho}$, and a $(\rho^{\mathcal{U}}, \rho')$-increasing permutation $\pi \in S_n$. Furthermore, if $\phi : \ca{A}_\rho \to \ca{A}_{\rho'}$ is a Jordan embedding, then there exists an invertible matrix $S \in \ca{A}_{\rho'}^\times$, a central idempotent $P \in Z(\ca{A}_\rho)$, a transitive map $g : \rho \to \C^\times$, and a $(\rho^{\mathcal{U}}, \rho')$-increasing permutation $\pi \in S_n$, where $\ca{U}:= \{1 \le i \le n : (i,i) \in \supp P\}$, such that
			$$\phi(\cdot) = (S R_{\pi}) (Pg^*(\cdot) + (I-P)g^*(\cdot)^t) (S R_{\pi})^{-1}.$$
		\end{corollary}
		\begin{proof}
			\fbox{$\implies$} Suppose that $\phi : \ca{A}_\rho \to \ca{A}_{\rho'}$ is a Jordan embedding. By Theorem \ref{thm:general form of phi}, there exists an invertible matrix $T \in M_n^\times$, a central idempotent $P \in Z(\ca{A}_\rho)$ and a transitive map $g : \rho \to \C^\times$ such that
			$$\phi(\cdot) = T (Pg^*(\cdot) + (I-P)g^*(\cdot)^t) T^{-1}.$$
			Denote $\Lambda_n:=\diag(1, \ldots, n) \in \ca{D}_n$. Obviously, the matrix $\phi(\Lambda_n) = T\Lambda_n T^{-1}$
			has eigenvalues $1, \ldots, n$ so by Theorem \ref{thm:inner diagonalization on SMA}, there exists $S\in \mathcal{A}_{\rho'}^\times$ and a permutation $\pi \in S_n$ such that $\phi(\Lambda_n) = (SR_{\pi})\Lambda_n (SR_{\pi})^{-1}$. We have 
			$$T\Lambda_n T^{-1} = (SR_{\pi})\Lambda_n (SR_{\pi})^{-1} \implies (SR_{\pi})^{-1} T \leftrightarrow \Lambda_n$$
			and hence we conclude that there exists a diagonal matrix $D \in \mathcal{D}_n^\times$ such that $$T = SR_{\pi} D = S\underbrace{(R_{\pi} D R_{\pi}^{-1})}_{\in \mathcal{D}_n^\times} R_{\pi}.$$
			By absorbing $R_{\pi} D R_{\pi}^{-1}$ into $S$, without loss of generality we can write $T = SR_{\pi}$ and therefore 
			$$\phi(\cdot) = (S R_{\pi}) (Pg^*(\cdot) + (I-P)g^*(\cdot)^t) (S R_{\pi})^{-1}.$$
			Note that, by Remark \ref{re:central decomposition}, $\ca{U}= \{1 \le i \le n : (i,i) \in \supp P\}$ is a union of certain classes of $\ca{Q}_{\rho}$. Consider the quasi-order $\rho^{\ca{U}}$ from \eqref{eq:wierd quasi order} with respect to this $\mathcal{U}$.  We in particular obtain $R_{\pi} \mathcal{A}_{\rho^{\ca{U}}} R_{\pi}^{-1} \subseteq \mathcal{A}_{\rho'}$, thus concluding that $\pi$ is $(\rho^{\mathcal{U}}, \rho')$-increasing.
			
			\smallskip
			
			\noindent \fbox{$\impliedby$} We have $R_{\pi} \mathcal{A}_{\rho^{\ca{U}}} R_{\pi}^{-1} \subseteq \mathcal{A}_{\rho'}$ for some subset $\ca{U} \subseteq \{1,\ldots,n\}$ expressible as a union of classes of $\ca{Q}_{\rho}$. Let $P \in Z(\ca{A}_{\rho})$ be a central idempotent corresponding to $\mathcal{U}$, i.e.\ $P := \sum_{i \in \ca{U}} E_{ii}$. It easily follows that
			$$R_{\pi} (P(\cdot) + (I-P)(\cdot)^t) R_{\pi}^{-1}$$
			is a Jordan embedding $\ca{A}_\rho \to \ca{A}_{\rho'}$.
		\end{proof}
		
		By plugging $\rho' = \rho$ into Corollary \ref{cor:SMA Jordan embeddings}, we obtain the description of Jordan automorphisms of SMAs:
		\begin{corollary}\label{cor:JAUT-SMA}
			Let $\ca{A}_\rho\subseteq M_n$ be an SMA. A map $\phi : \ca{A}_\rho \to \ca{A}_\rho$ is a Jordan automorphism if and only if there exists an invertible matrix $S \in \ca{A}_{\rho}^\times$, a central idempotent $P \in Z(\ca{A}_\rho)$, a transitive map $g : \rho \to \C^\times$, and a $(\rho^{\mathcal{U}}, \rho)$-increasing permutation $\pi \in S_n$, where $\ca{U}:= \{1 \le i \le n : (i,i) \in \supp P\}$, such that
			$$\phi(\cdot) = (S R_{\pi}) (Pg^*(\cdot) + (I-P)g^*(\cdot)^t) (S R_{\pi})^{-1}.$$
		\end{corollary}
		
		\smallskip
		
		At the end of this section, we also provide a complete answer to the first part of \cite[Problem 1.2]{GogicPetekTomasevic} for the class of SMAs. More precisely, we provide necessary and sufficient conditions for an SMA $\mathcal{A}_\rho \subseteq M_n$ such that each Jordan embedding $\phi : \mathcal{A}_\rho \to M_n$ extends to a Jordan automorphism of $M_n$.
		\begin{corollary}
			Let $\ca{A}_\rho\subseteq M_n$ be an SMA. The following conditions are equivalent:
			\begin{enumerate}[(i)]
				\item Each Jordan embedding $\phi : \mathcal{A}_\rho \to M_n$ extends to a Jordan automorphism of $M_n$.
				\item All transitive maps $g : \rho \to \C^\times$ are trivial and there is at most one class $C \in \ca{Q}$ such that $\abs{C} \ge 2$.
			\end{enumerate}
		\end{corollary}
		\begin{proof}
			This follows directly from Theorem \ref{thm:general form of phi}, Lemma \ref{le:g trivial iff induced map is inner automorphism} and Corollary \ref{cor:MAMP equivalences}.
		\end{proof}
		
		\section{Rank-one preserver}\label{sec:rank-one preserver}
		
		\subsection{On transitive maps and rank-one preservers} 
		
		We start this section by examining the connection between transitive maps and rank-one preservers on SMAs. We first state two elementary and well-known facts regarding the representation of rank-one matrices.   
		
		\begin{lemma}\label{le:uniqueness of rank-one representation}
			\phantom{x}
			\begin{enumerate}[(a)]
				\item A matrix $A \in M_n$ is rank-one if and only if it can be represented as $A=uv^*$ for some nonzero vectors $u,v \in \C^n$.
				\item Suppose that a rank-one matrix $A\in M_n$ has two distinct representations $A = u_1v_1^*= u_2v_2^*$ where $u_1, u_2, v_1, v_2 \in \C^n$. Then $u_1 \parallel u_2$ and $v_1 \parallel v_2$.
				\item  Let $u_1, u_2, v_1, v_2 \in \C^n$ be nonzero vector such that 
				$u_1v_1^* + u_2v_2^*$ is a rank-one matrix. Then $u_1 \parallel u_2$ or $v_1 \parallel v_2$.
			\end{enumerate}
		\end{lemma}
		
		As before, we make use of relations $\tripprox_0$ and $\tripprox$ (which are defined in Remark \ref{re:central decomposition}).
		\begin{lemma}\label{le:extending g trivially}
			Let $\ca{A}_\rho \subseteq M_n$ be an SMA and let $g : \rho \to \C^\times$ be a transitive map. Suppose that $$g|_{\rho \cap \{1,\ldots,n-1\}^2} \equiv 1.$$
			If $\mathop{\tripprox} \subseteq \{1,\ldots,n-1\}^2$ denotes the equivalence relation corresponding to the quasi-order $\rho \cap \{1,\ldots,n-1\}^2$, then we have the equivalence:
			$$\text{$g$ is trivial} \iff
			\begin{cases}
				(\forall i,j \in (\rho^\times)^{-1}(n))(i \tripprox j  \implies g(i,n) = g(j,n)),\\
				(\forall i,j \in (\rho^\times)(n))(i \tripprox j  \implies g(n,i) = g(n,j)).
			\end{cases}$$
		\end{lemma}
		\begin{proof}
			\fbox{$\implies$} Suppose that $g$ is trivial and separates through the map $s : \{1,\ldots,n\} \to \C^\times$. Note that for all $(i,j) \in \rho \cap \{1,\ldots,n-1\}^2$ we have $$1 = g(i,j) =\frac{s(i)}{s(j)} \implies s(i) =s(j).$$
			In particular, for all $1 \le i,j \le n-1$ we conclude 
			$$i \tripprox_0 j \implies s(i) = s(j)$$
			and then inductively
			\begin{equation}\label{eq:s is constant on equivalence classes}
				i \tripprox j \implies s(i) = s(j).
			\end{equation}
			Let $i,j \in (\rho^\times)^{-1}(n)$ such that $i \tripprox j$. We claim that $g(i,n) = g(j,n)$. We have
			$$g(i,n) = \frac{s(i)}{s(n)} \stackrel{\eqref{eq:s is constant on equivalence classes}}= \frac{s(j)}{s(n)} = g(j,n).$$
			This proves the first implication. The second implication is proved similarly.
			
			\smallskip
			
			\fbox{$\impliedby$} Suppose that the two implications hold true. Define $s : \{1,\ldots,n\} \to \C^\times$ by first setting $s(n) := 1$. For $1 \le j \le n-1$, let $[j]$ be the equivalence class of $\tripprox$ containing $j$, and set
			$$s(j) := \begin{cases}
				g(i,n), &\text{ if $i \in [j] \cap (\rho^\times)^{-1}(n)$},\\
				\frac1{g(n,i)}, &\text{ if $i \in [j] \cap (\rho^\times)(n)$},\\
				1, &\text{ if $[j] \cap (\rho^\times)^{-1}(n) = [j] \cap (\rho^\times)(n) = \emptyset$}.
			\end{cases}$$
			Note that $s$ is well-defined. Indeed, if $i_1,i_2 \in [j] \cap (\rho^\times)^{-1}(n)$, we have $i_1 \tripprox i_2$ and therefore $g(i_1,n) = g(i_2,n)$. The case $i_1,i_2 \in [j] \cap (\rho^\times)(n)$ is similar. Suppose now that $i_1 \in [j] \cap (\rho^\times)^{-1}(n)$ and also $i_2 \in [j] \cap (\rho^\times)(n)$. Then $(i_1,n),(n,i_2) \in \rho$ imply $(i_1,i_2) \in \rho$ and hence by transitivity
			$$g(i_1,n)g(n,i_2) = g(i_1,i_2) = 1 \implies g(i_1,n) = \frac1{g(n,i_2)},$$
			which is exactly what we wanted to show.\smallskip
			
			In particular, $s$ is constant on each equivalence class of $\tripprox$. Now we prove that $g$ separates through $s$. If $(i,j) \in \rho \cap \{1,\ldots,n-1\}^2$, then clearly $i \tripprox_0 j$ and hence $s(i) = s(j)$ which implies
			$$g(i,j) = 1 = \frac{s(i)}{s(j)}.$$
			If $(j,n) \in \rho^\times$, then $s(j) = g(j,n)$ by definition and therefore
			$$g(j,n) = \frac{s(j)}{s(n)}.$$
			Similarly we cover the case $(n,j) \in \rho^\times$.
		\end{proof}
		\begin{lemma}\label{le:rank one minors criterion}
			Let $X \in M_n$. Then $X$ has rank one if and only if $X \ne 0$ and for all $1 \le i,j,k,l \le n$ with $i \ne k$, $j \ne l$ holds
			$$\begin{vmatrix}
				X_{ij} & X_{il} \\ X_{kj} & X_{kl}
			\end{vmatrix} = 0.$$
		\end{lemma}
		\begin{proof}
			See for example the section on Minors and cofactors in \cite{Prasolov}.
		\end{proof}
		
		
		Let $\rho$ be a quasi-order on $\{1,\ldots,n\}$. Let $1 \le i,j,k,l \le n$. We say that $(i,j), (i,l),(k,j),(k,l)\in \rho$ form a \emph{rectangle} of the SMA $\ca{A}_\rho$ if $i \ne k$, $j \ne l$.
		
		\begin{lemma}\label{le:rank preserver rectangles}
			Let $\rho$ be a quasi-order on $\{1,\ldots,n\}$ and let $g : \rho \to \C^\times$ be a transitive map. Then the induced automorphism $g^* : \ca{A}_\rho \to \ca{A}_\rho$ is a rank-one preserver if and only if for every rectangle $(i,j), (i,l),(k,j),(k,l) \in \rho$ we have
			\begin{equation}\label{eq:g-minor is zero}
				\begin{vmatrix}
					g(i,j) & g(i,l) \\ g(k,j) & g(k,l)
				\end{vmatrix} = 0.
			\end{equation}
		\end{lemma}
		\begin{proof}
			\fbox{$\implies$} Suppose that $g^*$ is a rank-one preserver and let  $(i,j), (i,l),(k,j),(k,l) \in \rho$ be a rectangle of $\ca{A}_\rho$. The matrix
			$$g^*(E_{ij} + E_{il} + E_{kj} + E_{kl}) = g(i,j)E_{ij} + g(i,l)E_{il} + g(k,j)E_{kj} + g(k,l)E_{kl}$$
			has rank one, which by Lemma \ref{le:rank one minors criterion} implies the desired result.\smallskip
			
			\fbox{$\impliedby$} Conversely, suppose that $g$ satisfies the stated condition and let $X\in\ca{A}_\rho$ be a rank-one matrix. We wish to prove that $g^*(X)$ is a rank-one matrix. Since clearly $g^*(X) \ne 0$, by Lemma \ref{le:rank one minors criterion} it remains to verify the determinant condition. Fix $1 \le i,j,k,l \le n$ where $i \ne k$, $j \ne l$. If $\{(i,j), (i,l),(k,j),(k,l)\} \subsetneq \rho$, then the submatrix
			$\begin{bmatrix}
				X_{ij} & X_{il} \\ X_{kj} & X_{kl}
			\end{bmatrix}$
			has at least one zero-row or zero-column, so the same holds for $\begin{bmatrix}
				g^*(X)_{ij} & g^*(X)_{il} \\ g^*(X)_{kj} & g^*(X)_{kl}
			\end{bmatrix}.$ On the other hand, if $(i,j), (i,l),(k,j),(k,l)$ is a rectangle of $\rho$, we have
			$$\begin{vmatrix}
				X_{ij} & X_{il} \\ X_{kj} & X_{kl}
			\end{vmatrix} = 0 \implies X_{ij}X_{kl} = X_{il}X_{kj}$$
			and therefore
			\begin{align*}
				\begin{vmatrix}
					g^*(X)_{ij} & g^*(X)_{il} \\ g^*(X)_{kj} & g^*(X)_{kl}
				\end{vmatrix} &=
				\begin{vmatrix}
					g(i,j)X_{ij} & g(i,l)X_{il} \\ g(k,j)X_{kj} & g(k,l)X_{kl}
				\end{vmatrix} = \begin{vmatrix}
					g(i,j) & g(i,l) \\ g(k,j) & g(k,l)
				\end{vmatrix}(X_{ij}X_{kl}) \stackrel{\eqref{eq:g-minor is zero}}= 0,
			\end{align*}
			which completes the proof.
		\end{proof}
		
		\begin{corollary}
			Let $\ca{A}_\rho \subseteq M_n$ be an SMA without rectangles. Suppose that $g : \rho \to \C^\times$ is a transitive map. Then the induced automorphism $g^* : \ca{A}_\rho \to \ca{A}_\rho$ is a rank-one preserver.
		\end{corollary}

		\begin{lemma}\label{le:rank of canonical Jordan homomorphism}
			Let $\ca{A}_{\rho} \subseteq M_n$ be an SMA. For every central idempotent $P \in Z(\ca{A}_\rho)$ and $X \in \ca{A}_\rho$ we have the following equality of ranks:
			$$r(X) = r(PX + (I-P)X^t).$$
		\end{lemma}
		\begin{proof}
			It is easy to show that if two matrices $A,B \in M_n$ are supported on mutually disjoint sets of rows or columns, then $$r(A+B) = r(A)+r(B).$$
			Having this in mind, since $Z(\ca{A}_{\rho}) \subseteq \ca{D}_n$, we have $(I-P)^t = I-P \in Z(\ca{A}_{\rho})$. We obtain
			\begin{align*}
				r(X) &= r(PX + (I-P)X) = r(PX) + r((I-P)X) = r(PX) + r((I-P)X^t) \\
				&= r(PX + (I-P)X^t).
			\end{align*}
		\end{proof}
		\subsection{Main result}
		
		We are now ready to prove the main result of this section.
		
		\begin{theorem}\label{thm:rank-one preserver}
			Let $\ca{A}_\rho \subseteq M_n$ be an SMA.
			\begin{enumerate}[(a)]
				\item Let $\phi : \ca{A}_\rho \to M_n$ be a linear unital map preserving rank-one matrices. Then $\phi$ is a Jordan embedding.
				\item Let $\phi : \ca{A}_\rho \to M_n$ be a Jordan homomorphism which satisfies $\phi(E_{ij}) \ne 0$ for all $(i,j) \in \rho$. Then $\phi$ is a rank-one preserver if and only if the associated transitive map $g : \rho \to \C^\times$ obtained from Theorem \ref{thm:general form of phi} satisfies
				$$
				\begin{vmatrix}
					g(i,j) & g(i,l) \\ g(k,j) & g(k,l)
				\end{vmatrix} = 0
				$$
				for every rectangle $(i,j), (i,l),(k,j),(k,l)$ of $\ca{A}_\rho$.
			\end{enumerate}
		\end{theorem}
		The following remark justifies the claim of Theorem \ref{thm:rank-one preserver} (b).
		\begin{remark}
			If $\phi : \ca{A}_\rho \to M_n$ is a Jordan embedding, then the associated transitive map $g : \rho \to \C^\times$ obtained from Theorem \ref{thm:general form of phi} is not unique. However, if $g, h : \rho \to \C^\times$ are two such maps, then there exists a function $s : \{1,\ldots,n\} \to \C^\times$ such that $h(i,j) = \frac{s(i)}{s(j)}g(i,j)$ for all $(i,j) \in \rho$. Indeed, suppose that
			$$S(Pg^*(\cdot) + (I-P)g^*(\cdot)^t)S^{-1}  = T(Qh^*(\cdot) + (I-Q)h^*(\cdot)^t)T^{-1}$$
			for some $S,T \in M_n^\times$, central idempotents $P,Q \in Z(\ca{A}_\rho)$. Denote $\Lambda_n :=\diag(1,\ldots,n) \in \ca{D}_n$ and note that
			\begin{align*}
				S\Lambda_n S^{-1}&= S(Pg^*(\Lambda_n) + (I-P)g^*(\Lambda_n)^t)S^{-1}  = T(Qh^*(\Lambda_n) + (I-Q)h^*(\Lambda_n)^t)T^{-1} \\
				&= T\Lambda_n T^{-1},
			\end{align*}
			which implies $T^{-1}S \leftrightarrow \Lambda_n$ and hence $S = TD$ for some diagonal matrix $D \in \ca{D}_n^\times$. By plugging this in and cancelling $T$ from both sides, we obtain
			$$D(P(g^*(\cdot) + (I-P)g^*(\cdot)^t)D^{-1}  = Qh^*(\cdot) + (I-Q)h^*(\cdot)^t.$$
			Now let $(i,j) \in \rho^\times$ be arbitrary. Since $P$ and $Q$ are central, by Remark \ref{re:central decomposition} we certainly have $P_{ii} = P_{jj}$ and $Q_{ii} = Q_{jj}$. We consider four cases:
			\begin{itemize}
				\item If $P_{ii} = P_{jj} = Q_{ii} = Q_{jj} = 1$, then we obtain
				$$\frac{D_{ii}}{D_{jj}}g(i,j)E_{ij} = D(g(i,j)E_{ij})D^{-1} = h(i,j)E_{ij} \implies h(i,j) = \frac{D_{ii}}{D_{jj}}g(i,j).$$
				\item If $P_{ii} = P_{jj} = 1$ and $Q_{ii} = Q_{jj} = 0$, then we obtain
				$$\frac{D_{ii}}{D_{jj}}g(i,j)E_{ij} = D(g(i,j)E_{ij})D^{-1} =(h(i,j)E_{ij})^t = h(i,j)E_{ji},$$
				which is a contradiction. 
				\item The same argument also shows that the case $P_{ii} = P_{jj} = 0$ and $Q_{ii} = Q_{jj} = 1$ is also not possible.
				\item If $P_{ii} = P_{jj} = Q_{ii} = Q_{jj} = 0$, then we obtain
				$$\frac{D_{jj}}{D_{ii}}g(i,j)E_{ji} = D(g(i,j)E_{ij})^t D^{-1} = (h(i,j)E_{ij})^t = h(i,j)E_{ji} \implies h(i,j) = \frac{D_{jj}}{D_{ii}}g(i,j).$$
				Overall, if we define $s : \{1,\ldots,n\} \to \C^\times$ by $$
				s(i) :=\begin{cases}  D_{ii}, & \quad \mbox{ if } P_{ii}=1,\\
					\frac{1}{D_{ii}}, &  \quad \mbox{ otherwise}, 
				\end{cases}
				$$ 
				we conclude
				$$h(i,j) = \frac{s(i)}{s(j)}g(i,j), \qquad \text{ for all }(i,j) \in \rho.$$
				In particular, 
				$$\begin{vmatrix}
					h(i,j) & h(i,l) \\ h(k,j) & h(k,l)
				\end{vmatrix} = \begin{vmatrix}
					\frac{s(i)}{s(j)}g(i,j) & \frac{s(i)}{s(l)} g(i,l) \\ \frac{s(k)}{s(j)} g(k,j) & \frac{s(k)}{s(l)} g(k,l)
				\end{vmatrix} = \underbrace{\frac{s(i)s(k)}{s(j)s(l)}}_{\ne 0}\begin{vmatrix}
					g(i,j) & g(i,l) \\ g(k,j) & g(k,l)
				\end{vmatrix},$$
				which shows that the condition from Theorem \ref{thm:rank-one preserver} (b) is unambiguously defined (i.e.\ it is independent of the choice of the particular transitive map). 
				
			\end{itemize}
		\end{remark}
		\begin{proof}[Proof of Theorem \ref{thm:rank-one preserver}]
			First we prove $(a)$. By Lemma \ref{le:uniqueness of rank-one representation}, for each $(i,j) \in \rho$ we can choose $u_{ij}, v_{ij} \in \C^n$ such that $\phi(E_{ij}) = u_{ij}v_{ij}^*$.
			\begin{claim}\label{cl:ljiljana}\phantom{x}
				\begin{enumerate}[(a)]
					\item For all $1 \le i \le n$ and distinct $j,k \in \rho(i)$ we have either $u_{ij} \parallel u_{ik}$ or $v_{ij} \parallel v_{ik}$.
					\item For all $1 \le i \le n$ we have 
					$$\dim \spn\{u_{ij} : j \in \rho(i)\} = 1 \qquad \text{ or }\qquad \dim \spn\{v_{ij} : j \in \rho(i)\} = 1.$$
					If $(\rho^\times)(i)$ is nonempty, then the disjunction is exclusive.
					\item For all $1 \le i \le n$ we have the implications
					\begin{itemize}
						\item $\dim\spn\{u_{ij} : j \in \rho(i)\} = 1 \implies \{v_{ij} : j \in \rho(i)\}$ is linearly independent in $\C^n$.
						\item $\dim\spn\{v_{ij} : j \in \rho(i)\} = 1 \implies \{u_{ij} : j \in \rho(i)\}$ is linearly independent in $\C^n$.
					\end{itemize}
				\end{enumerate}
				These are rowwise versions; the columnwise versions of the claims also hold true:
				\begin{enumerate}[(a')]
					\item For all $1 \le j \le n$ and distinct $i,k \in \rho^{-1}(j)$ we have either $u_{ij} \parallel u_{kj}$ or $v_{ij} \parallel v_{kj}$.
					\item For all $1 \le j \le n$ we have 
					$$\dim\spn\{u_{ij} : i \in \rho^{-1}(j)\} = 1 \qquad \text{ or }\qquad \dim \spn\{v_{ij} : i \in \rho^{-1}(j)\} = 1.$$
					If $(\rho^\times)^{-1}(j)$ is nonempty, then the disjunction is exclusive.
					\item For all $1 \le i \le n$ we have the implications
					\begin{itemize}
						\item $\dim\spn\{u_{ij} : i \in \rho^{-1}(j)\} = 1 \implies \{v_{ij} : i \in \rho^{-1}(j)\}$ is linearly independent in $\C^n$.
						\item $\dim\spn\{v_{ij} : i \in \rho^{-1}(j)\} = 1 \implies \{u_{ij} : i \in \rho^{-1}(j)\}$ is linearly independent in $\C^n$.
					\end{itemize}
				\end{enumerate}
			\end{claim}
			We only prove (a), (b) and (c), as the proofs of (a'), (b') and (c') are analogous (or one can just pass to the map $\phi(\cdot)^t$).
			\begin{enumerate}
				\item[(a)] The matrices $\phi(E_{ij})$ and $\phi(E_{ik})$ are rank-one and their sum $\phi(E_{ij}+E_{ik})$ is rank-one as well so Lemma \ref{le:uniqueness of rank-one representation} (c) applies. Suppose now that $u_{ij} \parallel u_{ik}$ and $v_{ij} \parallel v_{ik}$. Let $\alpha,\beta \in \C^\times$ such that $u_{ik} = \alpha u_{ij}$ and $v_{ik} = \beta v_{ij}$. Then $$\phi(\alpha\overline{\beta} E_{ij} -E_{ik}) = \alpha\overline{\beta} u_{ij}v_{ij}^* - u_{ik}v_{ik}^* = \alpha\overline{\beta} u_{ij}v_{ij}^* - \alpha\overline{\beta} u_{ij}v_{ij}^* = 0,$$
				which is a contradiction.
				
				\item[(b)] For a fixed $1 \le i \le n$, we need to show that the same option from $(a)$ arises for all $1 \le j \ne k \le n$ such that $(i,j), (i,k) \in \rho$. Suppose the contrary, for example that $1 \le j,k,l \le n$ are distinct indices such that $(i,j), (i,k),(i,l) \in \rho$ and  $u_{ij} \parallel u_{ik} \not\parallel u_{il}.$ Then by $(a)$ it follows that $v_{ij} \not\parallel v_{ik} \parallel v_{il}$ and hence we have $u_{ij} \not\parallel u_{il}$ and $v_{ij} \not\parallel v_{il}$,  which is a contradiction with $(a)$. This proves the first part of the claim. To prove the second part, assume that $(i,j) \in \rho^\times$, but both dimensions are $1$. Then $u_{ii} \parallel u_{ij}$ and $v_{ii} \parallel v_{ij}$, which is a contradiction with $(a)$.
				
				\item[(c)] We will prove the first claim as the second one is very similar. Fix $1 \le i \le n$ and assume $\dim\spn\{u_{ij} : j \in \rho(i)\} = 1$. For each $j \in \rho(i)$ we can choose a scalar $\lambda_{ij} \in \C^\times$ such that $u_{ij} = \lambda_{ij}u_{ii}$. Then we have
				$$\phi(E_{ij}) = u_{ij}v_{ij}^* = (\lambda_{ij}u_{ii})v_{ij}^* = u_{ii}(\overline{\lambda_{ij}}v_{ij})^*,$$
				so we can replace $v_{ij}$ with $\overline{\lambda_{ij}}v_{ij}$ and assume that $\phi(E_{ij}) = u_{ii}v_{ij}^*$ for all $j \in \rho(i)$. Let $\alpha_{ij} \in \C, j \in \rho(i)$ be scalars such that $$\sum_{j \in \rho(i)}\alpha_{ij}v_{ij} = 0.$$
				Then $$\phi\left(\sum_{j \in \rho(i)}\alpha_{ij}E_{ij}\right) = u_{ii}\sum_{j \in \rho(i)}\alpha_{ij}v_{ij}^* = 0.$$
				This implies $\alpha_{ij}= 0$ for all $j \in \rho(i)$, as otherwise the rank-one matrix $\sum_{j \in \rho(i)}\alpha_{ij}E_{ij}$ is in the kernel of $\phi$.

			\end{enumerate}
			
			\begin{claim}\label{cl:diagonal elements are linearly 
					independent} The sets $\{u_{11}, \ldots, u_{nn}\}$ and $\{v_{11}, \ldots, v_{nn}\}$ are linearly independent. Furthermore, they satisfy the orthogonality relations
				$$\inner{u_{ii}}{v_{jj}} = \delta_{ij}, \qquad 1\le i,j \le n.$$
			\end{claim}
			
			We have $$I = \phi(I) = \sum_{i=1}^n u_{ii}v_{ii}^* = \begin{bmatrix}u_{11} & \cdots & u_{nn}\end{bmatrix}\begin{bmatrix}v_{11} & \cdots & v_{nn}\end{bmatrix}^*.$$
			In particular, the matrices $\begin{bmatrix}u_{11} & \cdots & u_{nn}\end{bmatrix}$ and $\begin{bmatrix}v_{11} & \cdots & v_{nn}\end{bmatrix}$ are invertible. We also have
			$$[\delta_{ij}]_{i,j=1}^n = I = \begin{bmatrix}v_{11} & \cdots & v_{nn}\end{bmatrix}^*\begin{bmatrix}u_{11} & \cdots & u_{nn}\end{bmatrix} = [v_{ii}^*u_{jj}]_{i,j=1}^n = [\inner{u_{jj}}{v_{ii}}]_{i,j=1}^n.$$
			
			Denote $$\ca{U}_R := \{1 \le i \le n : \dim \spn\{u_{ij} : j \in \rho(i)\} = 1\},$$
			$$\ca{V}_R := \{1 \le i \le n : \dim \spn\{v_{ij} : j \in \rho(i)\} 
			= 1\}.$$
			By Claim \ref{cl:ljiljana} (b), it is clear that
			\begin{equation}\label{eq:union and intersection rows}
				\ca{U}_R \cup \ca{V}_R = \{1, \ldots, n\}, \qquad \ca{U}_R \cap \ca{V}_R = \{1 \le i \le n : (\rho^\times)(i) = \emptyset\}.
			\end{equation}
			Similarly, for the columns, we define
			$$\ca{U}_C := \{1 \le j \le n : \dim \spn\{u_{ij} : i \in \rho^{-1}(j)\} = 1\},$$
			$$\ca{V}_C := \{1 \le j \le n : \dim \spn\{v_{ij} : i \in \rho^{-1}(j)\} = 1\}$$
			and we can make similar observations as above to conclude
			\begin{equation}\label{eq:union and intersection columns}
				\ca{U}_C \cup \ca{V}_C = \{1, \ldots, n\}, \qquad \ca{U}_C \cap \ca{V}_C = \{1 \le j \le n : \rho^{-1}(j) = \emptyset\}.
			\end{equation}
			
			\begin{claim}\label{cl:interlaced products}
				We have
				$$\rho^\times \subseteq ((\ca{U}_R\setminus \ca{V}_R) \times (\ca{V}_C\setminus \ca{U}_C)) \sqcup ((\ca{V}_R\setminus \ca{U}_R) \times (\ca{U}_C\setminus \ca{V}_C)),$$
				where $\sqcup$ stands for the disjoint union.
			\end{claim}
			
			Let $(i,j) \in \rho^\times$. Since $(\rho^\times)(i)$  and $(\rho^\times)^{-1}(j)$ are both nonempty, by \eqref{eq:union and intersection rows} and \eqref{eq:union and intersection columns} we have
			$$i \in (\ca{U}_R\setminus \ca{V}_R) \sqcup  (\ca{V}_R\setminus \ca{U}_R) \quad \text{ and } \quad j \in (\ca{U}_C\setminus \ca{V}_C) \sqcup  (\ca{V}_C\setminus \ca{U}_C).$$
			Suppose by way of contradiction that $(i,j) \in (\ca{U}_R\setminus \ca{V}_R)\times (\ca{U}_C\setminus \ca{V}_C)$. It follows
			$u_{ii} \parallel u_{ij} \parallel u_{jj},$
			which contradicts Claim \ref{cl:diagonal elements are linearly independent}. The case $(i,j) \in (\ca{V}_R\setminus \ca{U}_R)\times (\ca{V}_C\setminus \ca{U}_C)$ is similarly shown to be impossible.

			\begin{claim}\label{cl:transitive elements}
				Suppose that $(i,j)\in \ca{\rho}^\times$. If there exists $(j,k)\in \ca{\rho}^\times$ then
				$$(i,j) \in ((\ca{U}_R\setminus \ca{V}_R) \times (\ca{U}_R\setminus \ca{V}_R)) \sqcup ((\ca{V}_R\setminus \ca{U}_R) \times (\ca{V}_R\setminus \ca{U}_R)).$$
			\end{claim}
			
			As in the proof of Claim \ref{cl:interlaced products}, first note that $i,j \in (\ca{U}_R\setminus \ca{V}_R) \sqcup  (\ca{V}_R\setminus \ca{U}_R)$. We shall assume that $i \in \ca{U}_R\setminus \ca{V}_R$ as the other case is similar. Assume first that $k \ne i$. Then by Claim \ref{cl:interlaced products} from $(i,k) \in \rho^\times$ it follows
			$$i \in \ca{U}_R\setminus \ca{V}_R \implies k \in \ca{V}_C\setminus \ca{U}_C \implies j \in \ca{U}_R\setminus \ca{V}_R.$$
			Now assume that $k = i$. Then $(i,i),(i,j),(j,i),(j,j) \in \rho$ so the matrix \begin{align*}
				\phi(E_{ii} + E_{ij} + E_{ji} + E_{jj}) &= u_{ii}v_{ii}^* + u_{ij}v_{ij}^* + u_{ji}v_{ji}^* + u_{jj}v_{jj}^*
			\end{align*}
			has rank one. By \eqref{eq:union and intersection rows}, $j \in \ca{U}_R \setminus \ca{V}_R$ or $j \in \ca{V}_R\setminus \ca{U}_R$ so by way of contradiction suppose the latter. Then there exist scalars $\alpha,\beta \in \C^\times$ such that $u_{ij} = \alpha u_{ii}$ and $v_{ji} = \beta v_{jj}$. Therefore
			$$u_{ii}v_{ii}^* + u_{ij}v_{ij}^* + u_{ji}v_{ji}^* + u_{jj}v_{jj}^* = u_{ii}(v_{ii}+\overline{\alpha}v_{ij})^* + (\overline{\beta}u_{ji} + u_{jj})v_{jj}^*$$
			so by Lemma \ref{le:uniqueness of rank-one representation} (c) it follows that
			$$u_{ii} \parallel (\overline{\beta}u_{ji} + u_{jj}) \qquad \text{ or }\qquad \left(v_{ii}+\overline{\alpha}v_{ij}\right) \parallel v_{jj}.$$
			By Claim \ref{cl:interlaced products}, since $(j,i) \in \rho^\times$, from $j \in \ca{V}_R\setminus \ca{U}_R$ we obtain $i \in \ca{U}_C\setminus \ca{V}_C$ and therefore
			$$u_{ii} \parallel \left(\overline{\beta}\underbrace{u_{ji}}_{\parallel u_{ii}} + u_{jj}\right)$$
			is a contradiction with Claim \ref{cl:diagonal elements are linearly independent}. Similarly, from $i \in \ca{U}_R\setminus \ca{V}_R$ we obtain $j \in \ca{V}_C\setminus \ca{U}_C$ and therefore
			$$\left(v_{ii}+\overline{\alpha}\underbrace{v_{ij}}_{\parallel v_{jj}}\right) \parallel v_{jj}$$
			is a contradiction with Claim \ref{cl:diagonal elements are linearly independent}.

			\begin{claim}\label{cl:inner product scalars}
				Suppose that two nonzero vectors $v,w \in \C^n$ satisfy $v\parallel w$. Then
				$$v = \frac{\inner{v}{w}}{\norm{w}^2}w = \frac{\norm{v}^2}{\inner{w}{v}}w.$$
			\end{claim}
			
			This follows from an easy computation.

			\begin{claim}\label{cl:rectangle inner products}
				Suppose that $(p,q),(p,s),(r,q),(r,s)\in\rho$ form a rectangle in $\ca{A}_\rho$. Then
				$$\inner{u_{pq}}{u_{ps}}\inner{u_{rs}}{u_{rq}}\inner{v_{rq}}{v_{pq}}\inner{v_{ps}}{v_{rs}} = \norm{u_{pq}}^2\norm{u_{rs}}^2\norm{v_{pq}}^2\norm{v_{rs}}^2.$$
			\end{claim}
			
			By definition we have $p\ne r$ and $q \ne s$ which implies that all sets $(\rho^\times)(p)$, $(\rho^\times)(r)$, $(\rho^\times)^{-1}(q)$, $(\rho^\times)^{-1}(s)$ are nonempty, so by \eqref{eq:union and intersection rows} we may assume that $p \in \ca{U}_R\setminus \ca{V}_R$ (as the other case $p \in \ca{V}_R\setminus \ca{U}_R$ is similar). We first show that $r \in \ca{U}_R\setminus \ca{V}_R$ as well.
			If $r=q$, then in particular $r\ne s$ so $(p,r),(r,s) \in \rho^\times$ and hence $r \in \ca{U}_R\setminus \ca{V}_R$ follows from Claim \ref{cl:transitive elements}. Assume $r \ne q$.
			\begin{itemize}
				\item If $p =q$, then $r \ne p \ne s$ and therefore $(r,p),(p,s) \in \rho^\times$ so by Claim \ref{cl:transitive elements} it follows that $r \in \ca{U}_R\setminus \ca{V}_R$.
				\item If $p \ne q$, then $$p \in \ca{U}_R\setminus \ca{V}_R \stackrel{\text{Claim }\ref{cl:interlaced products}}\implies q \in \ca{V}_C\setminus \ca{U}_C \stackrel{\text{Claim }\ref{cl:interlaced products}}\implies r \in \ca{U}_R\setminus \ca{V}_R.$$
			\end{itemize}
			Putting it all together, since at least one of $(p,q),(r,q)$ is in $\rho^\times$, and similarly for $(p,s),(r,s)$, we have
			$$p,r \in \ca{U}_R\setminus \ca{V}_R \stackrel{\text{Claim }\ref{cl:interlaced products}}\implies q,s \in \ca{V}_C\setminus \ca{U}_C$$
			and hence, by Claim \ref{cl:inner product scalars},
			$$p \in \ca{U}_R\setminus \ca{V}_R \implies u_{ps} = \frac{\inner{u_{ps}}{u_{pq}}}{\norm{u_{pq}}^2}u_{pq},\qquad r \in \ca{U}_R\setminus \ca{V}_R \implies u_{rq} = \frac{\inner{u_{rq}}{u_{rs}}}{\norm{u_{rs}}^2}u_{rs},$$
			$$q \in \ca{V}_C\setminus \ca{U}_C \implies v_{rq} = \frac{\inner{v_{rq}}{v_{pq}}}{\norm{v_{pq}}^2}v_{pq},\qquad s \in \ca{V}_C\setminus \ca{U}_C \implies v_{rs} = \frac{\norm{v_{rs}}^2}{\inner{v_{ps}}{v_{rs}}}v_{ps}.$$
			The matrix
			\begin{align*}
				&\phi(E_{pq}+E_{ps}+E_{rq}+E_{rs}) = u_{pq}v_{pq}^* + u_{ps}v_{ps}^* + u_{rq}v_{rq}^* + u_{rs}v_{rs}^*\\
				&\qquad= u_{pq}\left(v_{pq} +\frac{\inner{u_{pq}}{u_{ps}}}{\norm{u_{pq}}^2}v_{ps}\right)^* + u_{rs}\left(\frac{\inner{u_{rs}}{u_{rq}}}{\norm{u_{rs}}^2} v_{rq} + v_{rs}\right)^*\\
				&\qquad= u_{pq}\left(v_{pq} + \frac{\inner{u_{pq}}{u_{ps}}}{\norm{u_{pq}}^2}v_{ps}\right)^* + u_{rs}\left(\frac{\inner{u_{rs}}{u_{rq}}}{\norm{u_{rs}}^2} \frac{\inner{v_{rq}}{v_{pq}}}{\norm{v_{pq}}^2}v_{pq} + \frac{\norm{v_{rs}}^2}{\inner{v_{ps}}{v_{rs}}}v_{ps}\right)^*
			\end{align*}
			has rank one. Since $q \in \ca{V}_C\setminus \ca{U}_C$ and $r \in \ca{U}_R\setminus \ca{V}_R$, by Claim \ref{cl:ljiljana} (a) we have $u_{pq} \not\parallel u_{rq} \parallel u_{rs}$. Claim \ref{le:uniqueness of rank-one representation} (c) now yields
			$$\left(v_{pq} + \frac{\inner{u_{pq}}{u_{ps}}}{\norm{u_{pq}}^2}v_{ps}\right) \parallel \left(\frac{\inner{u_{rs}}{u_{rq}}}{\norm{u_{rs}}^2} \frac{\inner{v_{rq}}{v_{pq}}}{\norm{v_{pq}}^2}v_{pq} + \frac{\norm{v_{rs}}^2}{\inner{v_{ps}}{v_{rs}}}v_{ps}\right).$$
			In particular, we have
			$$0 = \begin{vmatrix}
				1 & \frac{\inner{u_{pq}}{u_{ps}}}{\norm{u_{pq}}^2} \\ 
				\frac{\inner{u_{rs}}{u_{rq}}}{\norm{u_{rs}}^2} \frac{\inner{v_{rq}}{v_{pq}}}{\norm{v_{pq}}^2} & \frac{\norm{v_{rs}}^2}{\inner{v_{ps}}{v_{rs}}} = 
			\end{vmatrix} = \frac{\norm{v_{rs}}^2}{\inner{v_{ps}}{v_{rs}}} - \frac{\inner{u_{rs}}{u_{rq}}}{\norm{u_{rs}}^2} \frac{\inner{v_{rq}}{v_{pq}}}{\norm{v_{pq}}^2}\frac{\inner{u_{pq}}{u_{ps}}}{\norm{u_{pq}}^2}$$
			which is exactly what we desired to show.
			
			\smallskip
			
			In the next claim we prove that $\phi$ is a Jordan homomorphism.
			\begin{claim}
				For all $(i,j), (k,l) \in \rho$ we have
				\begin{equation}\label{eq:Jordan product}
					\phi(E_{ij} \circ E_{kl}) = \phi(E_{ij}) \circ \phi(E_{kl}).
				\end{equation}
			\end{claim}
			
			We will assume throughout that $i \in \ca{U}_R$ (as the case $i \in \ca{V}_R$ can be treated similarly). We have several cases to consider. 
			\begin{enumerate}
				\item[($1^\circ$)] Suppose $i = l = k = j$. Then
				$$\phi(E_{ii} \circ E_{ii}) = 2\phi(E_{ii})=  2u_{ii}v_{ii}^* \stackrel{\text{Claim } \ref{cl:diagonal elements are linearly independent}}= 2u_{ii}(v_{ii}^*u_{ii})v_{ii}^* = 2\phi(E_{ii})^2 = \phi(E_{ii}) \circ \phi(E_{ii})$$
				so \eqref{eq:Jordan product} holds.
				
				\item[($2^\circ$)] Suppose $i = l \ne k = j$. Then
				$$\phi(E_{ij} \circ E_{ji}) = \phi(E_{ii}) +\phi(E_{jj}) = u_{ii}v_{ii}^* + u_{jj}v_{jj}^*.$$
				On the other hand,               
				\begin{align*}
					\phi(E_{ij})\phi(E_{ji})+ \phi(E_{ji})\phi(E_{ij}) &= u_{ij}(v_{ij}^*u_{ji})v_{ji}^* + u_{ji}(v_{ji}^*u_{ij})v_{ij}^* \\
					&= \inner{u_{ji}}{v_{ij}} u_{ij}v_{ji}^* + \inner{u_{ij}}{v_{ji}} u_{ji}v_{ij}^*
				\end{align*}
				Since $(i,j), (j,i) \in \rho^\times$, by \eqref{eq:union and intersection rows} we have $i \in \ca{U}_R\setminus \ca{V}_R$ and hence
				\begin{align*}
					&i \in \ca{U}_R\setminus \ca{V}_R \stackrel{\text{Claim }\ref{cl:transitive elements}}\implies j \in \ca{U}_R\setminus \ca{V}_R \implies u_{ji} \parallel u_{jj} \stackrel{\text{Claim } \ref{cl:inner product scalars}}\implies u_{ji} = \frac{\inner{u_{ji}}{u_{jj}}}{\norm{u_{jj}}^2}u_{jj},\\
					&j \in \ca{U}_R\setminus \ca{V}_R \stackrel{\text{Claim }\ref{cl:interlaced products}}\implies i \in  \ca{V}_C\setminus \ca{U}_C \implies v_{ji} \parallel v_{ii} \stackrel{\text{Claim } \ref{cl:inner product scalars}}\implies v_{ji} = \frac{\inner{v_{ji}}{v_{ii}}}{\norm{v_{ii}}^2}v_{ii},\\
					&i \in \ca{U}_R\setminus \ca{V}_R \stackrel{\text{Claim }\ref{cl:interlaced products}}\implies j \in \ca{V}_C\setminus \ca{U}_C \implies v_{ij} \parallel v_{jj} \stackrel{\text{Claim } \ref{cl:inner product scalars}}\implies v_{ij} = \frac{\inner{v_{ij}}{v_{jj}}}{\norm{v_{jj}}^2}v_{jj}.
				\end{align*}
				We also have
				$$i \in \ca{U}_R\setminus \ca{V}_R \implies u_{ij} \parallel u_{ii} \stackrel{\text{Claim } \ref{cl:inner product scalars}}\implies u_{ij} = \frac{\inner{u_{ij}}{u_{ii}}}{\norm{u_{ii}}^2}u_{ii}.$$
				Therefore,
				\begin{align*}
					\inner{u_{ji}}{v_{ij}} u_{ij}v_{ji}^* &= \inner{\frac{\inner{u_{ji}}{u_{jj}}}{\norm{u_{jj}}^2}u_{jj}}{\frac{\inner{v_{ij}}{v_{jj}}}{\norm{v_{jj}}^2}v_{jj}} \left(\frac{\inner{u_{ij}}{u_{ii}}}{\norm{u_{ii}}^2}u_{ii}\right)\left(\frac{\inner{v_{ji}}{v_{ii}}}{\norm{v_{ii}}^2}v_{ii}\right)^*\\
					&\leftstackrel{\text{Claim } \ref{cl:diagonal elements are linearly independent}}= \frac{\inner{u_{ji}}{u_{jj}}\inner{v_{jj}}{v_{ij}}\inner{u_{ij}}{u_{ii}}\inner{v_{ii}}{v_{ji}}}{\norm{u_{jj}}^2\norm{v_{jj}}^2\norm{u_{ii}}^2\norm{v_{ii}}^2}  u_{ii}v_{ii}^*\\
					&= u_{ii}v_{ii}^*,
				\end{align*}
				where the last equality follows by conjugation from Claim \ref{cl:rectangle inner products} for $p=q=i$ and $r=s=j$. By exchanging $i$ and $j$ in the same way we obtain
				$$\inner{u_{ij}}{v_{ji}} u_{ji}v_{ij}^* = u_{jj}v_{jj}^*.$$
				This proves \eqref{eq:Jordan product}.

				\item[($3^\circ$)] Suppose $k = j$ but $i \ne l$. Then $(i,j), (j,l) \in\rho$ so $(i,l) \in \rho^\times$ (in particular, $i \in \ca{U}_R\setminus \ca{V}_R$ by \eqref{eq:union and intersection rows}). We need to prove that
				$$\phi(E_{ij}E_{jl} + E_{jl}E_{ij}) = \phi(E_{il}) = u_{il}v_{il}^*$$
				equals
				\begin{align*}
					\phi(E_{ij})\phi(E_{jl}) + \phi(E_{jl})\phi(E_{ij}) &= u_{ij}(v_{ij}^*u_{jl})v_{jl}^* + u_{jl}(v_{jl}^*u_{ij})v_{ij}^*\\
					&= \inner{u_{jl}}{v_{ij}}u_{ij}v_{jl}^* + \inner{u_{ij}}{v_{jl}}u_{jl}v_{ij}^*.
				\end{align*}
				For the time being, we additionally assume that $j \ne l$. First note that we always have $j \in \ca{U}_R\setminus \ca{V}_R$. Indeed, if $i = j$, then this is obvious, while if $i \ne j$ then $(i,j),(j,l) \in \rho^\times$, so this follows from Claim \ref{cl:transitive elements}. In any case, since $(j,l) \in \rho^\times$, Claim \ref{cl:interlaced products} also implies $l \in \ca{V}_C \setminus \ca{U}_C$. We now have 
				$$i \in \ca{U}_R\setminus \ca{V}_R \implies u_{ij} \parallel u_{ii},\qquad l \in \ca{V}_C\setminus \ca{U}_C \implies v_{jl} \parallel v_{ll}$$
				and hence $\inner{u_{ij}}{v_{jl}}=0$ by Claim \ref{cl:diagonal elements are linearly independent}. Therefore, the second term of the right hand size is zero. Now we focus on the first term.
				
				\smallskip
				
				We have
				$$j \in \ca{U}_R\setminus \ca{V}_R \implies u_{jl} \parallel u_{jj},\qquad l \in \ca{V}_C\setminus \ca{U}_C \implies v_{jl} \parallel v_{il},$$
				$$i\in \ca{U}_R\setminus \ca{V}_R \implies u_{ij} \parallel u_{il},\qquad (i = j \text{ or } (i \ne j \stackrel{\text{Claim } \ref{cl:interlaced products}}\implies j \in \ca{V}_C\setminus \ca{U}_C)) \implies v_{ij} \parallel v_{jj}$$
				
				and therefore (again invoking Claims \ref{cl:inner product scalars} and \ref{cl:diagonal elements are linearly independent}),
				\begin{align*}
					\inner{u_{jl}}{v_{ij}}u_{ij}v_{jl}^* &= \inner{\frac{\inner{u_{jl}}{u_{jj}}}{\norm{u_{jj}}^2} u_{jj}} {\frac{\inner{v_{ij}}{v_{jj}}}{\norm{v_{jj}}^2} v_{jj}}\left(\frac{\inner{u_{ij}}{u_{il}}}{\norm{u_{il}}^2} u_{il}\right)\left(\frac{\inner{v_{jl}}{v_{il}}}{\norm{v_{il}}^2} v_{il}\right)^*\\
					&= \frac{\inner{u_{jl}}{u_{jj}} \inner{v_{jj}}{v_{ij}} \inner{u_{ij}}{u_{il}} \inner{v_{il}}{v_{jl}}}{\norm{u_{jj}}^2\norm{v_{jj}}^2\norm{u_{il}}^2\norm{v_{il}}^2} u_{il}v_{il}^*\\
					&= u_{il}v_{il}^*,
				\end{align*}
				where the last equality is obtained by conjugating the equality from Claim \ref{cl:rectangle inner products} for $(p,q) = (i,l)$ and $(r,s) = (j,j)$. This proves \eqref{eq:Jordan product}.
				
				\smallskip
				
				It remains to consider the case $j = l$ so we are looking at $(i,l),(l,l)$. We need to prove that
				$$\phi(E_{il}E_{ll} + E_{ll}E_{il}) = \phi(E_{il}) = u_{il}v_{il}^*$$
				equals
				\begin{align}\label{eq:Eil and Ell}
					\phi(E_{il})\phi(E_{ll}) + \phi(E_{ll})\phi(E_{il}) &= u_{il}(v_{il}^*u_{ll})v_{ll}^* + u_{ll}(v_{ll}^*u_{il})v_{il}^*\notag \\
					&= \inner{u_{ll}}{v_{il}}u_{il}v_{ll}^* + \inner{u_{il}}{v_{ll}}u_{ll}v_{il}^*.
				\end{align}
				Since $i \in \ca{U}_R\setminus \ca{V}_R$, we have $u_{il}\parallel u_{ii}$ and thus $\inner{u_{il}}{v_{ll}} = 0$ by Claim \ref{cl:diagonal elements are linearly independent}, rendering the second term of \eqref{eq:Eil and Ell} zero. We have 
				$$i \in \ca{U}_R\setminus \ca{V}_R \stackrel{\text{Claim } \ref{cl:interlaced products}}\implies l \in \ca{V}_C\setminus \ca{U}_C \implies v_{ll} \parallel v_{il} $$
				and therefore (by Claims \ref{cl:inner product scalars} and \ref{cl:diagonal elements are linearly independent}), the first term of \eqref{eq:Eil and Ell} is
				
				$$\inner{u_{ll}}{v_{il}}u_{il}v_{ll}^* = \inner{u_{ll}}{\frac{\inner{v_{il}}{v_{ll}}}{\norm{v_{ll}}^2} v_{ll}}u_{il}\left(\frac{\norm{v_{ll}}^2}{\inner{v_{il}}{v_{ll}}} v_{il}\right)^* = u_{il}v_{il}^*$$
				which proves \eqref{eq:Jordan product}.
				
				\item[($4^\circ$)] Suppose $i = l$ but $k \ne j$. By the commutativity of the Jordan product, this case reduces to $(3^\circ)$ by exchanging $E_{ij}$ and $E_{kl}$.
				\item[($5^\circ$)] Suppose $k \ne j$ and $i \ne l$  (so, as before, $i \in \ca{U}_R\setminus \ca{V}_R$). Then the left hand side of \eqref{eq:Jordan product} is zero, while the right hand side equals
				\begin{align*}
					\phi(E_{ij})\phi(E_{kl}) + \phi(E_{kl})\phi(E_{ij}) &= u_{ij}(v_{ij}^* u_{kl})v_{kl}^* + u_{kl}(v_{kl}^* u_{ij})v_{ij}^* \\
					&= \inner{u_{kl}}{v_{ij}} u_{ij}v_{kl}^* + \inner{u_{ij}}{v_{kl}} u_{kl}v_{ij}^*.
				\end{align*}
				Note that
				\begin{equation}\label{eq:vij parallel vjj}
					(i = j \text{ or } (i \ne j \stackrel{\text{Claim } \ref{cl:interlaced products}}\implies j \in \ca{V}_C\setminus \ca{U}_C)) \implies v_{ij} \parallel v_{jj}.
				\end{equation}
				Suppose that $k \in \ca{U}_R$. Then
				$$(k = l \text{ or } (k \ne l \implies k \in \ca{U}_R\setminus\ca{V}_R \stackrel{\text{Claim } \ref{cl:interlaced products}}\implies l \in \ca{V}_C\setminus \ca{U}_C)) \implies v_{kl} \parallel v_{ll}.$$
				We also have
				$$i \in \ca{U}_R \implies u_{ii} \parallel u_{ij}, \qquad k \in \ca{U}_R \implies u_{kk} \parallel u_{kl},$$
				so $\inner{u_{kl}}{v_{ij}} = 0$ and  $\inner{u_{ij}}{v_{kl}} = 0$ by Claim \ref{cl:diagonal elements are linearly independent}.
				
				\smallskip
				
				Suppose now $k \in \ca{V}_R\setminus \ca{U}_R$. Then in particular $i 
				\ne k$. Assume first that $j \ne l$. We have
				$$i \in \ca{U}_R \implies u_{ii} \parallel u_{ij}, \qquad k \in \ca{V}_R \implies v_{kl} \parallel v_{kk},$$
				$$(k = l \text{ or } (k \ne l \stackrel{\text{Claim } \ref{cl:interlaced products}}\implies l \in \ca{U}_C\setminus \ca{V}_C)) \implies u_{ll} \parallel u_{kl} $$
				so by \eqref{eq:vij parallel vjj} we have $\inner{u_{kl}}{v_{ij}} = \inner{u_{ij}}{v_{kl}} = 0$ by Claim \ref{cl:diagonal elements are linearly independent}. Now assume $j = l$. Then $i \ne l = j$ so $(i,j), (k,j) \in \rho^\times$ and therefore
				$$i \in \ca{U}_R \setminus \ca{V}_R \stackrel{\text{Claim }\ref{cl:interlaced products}}\implies j \in \ca{V}_C \setminus \ca{U}_C \stackrel{\text{Claim }\ref{cl:interlaced products}}\implies k \in \ca{U}_R \setminus \ca{V}_R,$$
				which is a contradiction. Therefore $j = l$ is impossible.
				
				\smallskip
				
				In either case, this proves $\phi(E_{ij})\phi(E_{kl}) + \phi(E_{kl})\phi(E_{ij}) = 0$.
				
			\end{enumerate}
			
			\begin{claim}
				$\phi$ is injective.
			\end{claim}
			
			Indeed, $\phi$ is a Jordan homomorphism which clearly satisfies $\phi(E_{ij}) \ne 0$ for all $(i,j) \in \rho$. Now Theorem \ref{thm:general form of phi} directly implies the claim.
			
			\smallskip

			This concludes the proof of (a). Now we prove (b). In view of Lemma \ref{le:rank preserver rectangles}, it suffices to prove that $\phi$ is a rank-one preserver if and only if $g^*$ is a rank-one preserver.
			
			\smallskip
			
			Theorem \ref{thm:general form of phi} implies that there exists an invertible matrix $T \in M_n^\times$, a central idempotent $P \in Z(\ca{A}_\rho)$, and a transitive map $g : \rho \to \C^\times$ such that
			$$\phi(\cdot) = T(Pg^*(\cdot) + (I-P)g^*(\cdot)^t)T^{-1}.$$
			For each $X \in \ca{A}_{\rho}$ we have
			\begin{align*}
				r(X) &= r(\phi(X)) = r(T(Pg^*(X) + (I-P)g^*(X)^t)T^{-1}) = r(Pg^*(X) + (I-P)g^*(X)^t) \\
				&\leftstackrel{\text{Lemma } \ref{le:rank of canonical Jordan homomorphism}} = r(g^*(X))
			\end{align*}
			This implies that $\phi(X)$ is a rank-one matrix if and only if $g^*(X)$ is. Since $X \in \ca{A}_\rho$ was an arbitrary rank-one matrix, the claim follows.
		\end{proof}
		
		\begin{remark}\label{rem:nontrivial transitive map}
			For a concrete example showing that the converse of Theorem \ref{thm:rank-one preserver} (a) does not hold in general, consider the SMA $\ca{A}_\rho \subseteq M_4$, where 
			$$\rho := \Delta_4 \cup \{(1,3),(1,4),(2,3),(2,4)\}$$
			and the transitive map $$g : \rho \to \C^\times, \qquad g(i,j) = \begin{cases}
				1, \quad &\text{ if } (i,j) \ne (1,4),\\
				2, \quad &\text{ if } (i,j)=(1,4).
			\end{cases}$$
			Then the induced automorphism
			$$g^* : \ca{A}_{\rho} \to \ca{A}_\rho, \qquad \begin{bmatrix}
				x_{11} & 0 & x_{13} & x_{14} \\
				0 & x_{22} & x_{23} & x_{24} \\
				0 & 0 & x_{33} & 0 \\
				0 & 0 & 0 & x_{44}
			\end{bmatrix} \mapsto \begin{bmatrix}
				x_{11} & 0 & x_{13} & 2x_{14} \\
				0 & x_{22} & x_{23} & x_{24} \\
				0 & 0 & x_{33} & 0 \\
				0 & 0 & 0 & x_{44}
			\end{bmatrix}$$
			is clearly not a rank-one preserver. \smallskip
			
			However, by direct computations, one easily shows that the converse of Theorem \ref{thm:rank-one preserver} (a) does hold if $n \le 3$. Moreover, for general $n \in \N$ and an SMA $\ca{A}_\rho \subseteq M_n$, the same  holds true if we only assume $\abs{C} \le 3$ for all $C \in \ca{Q}$ (following the notation from Section 3). It can be easily shown that this condition, combined  with the  central decomposition of $\mathcal{A}_\rho$ from Remark \ref{re:central decomposition},  implies that all transitive maps $g: \rho \to \C^\times$ are  necessarily trivial. Moreover, using Lemma \ref{le:rank of canonical Jordan homomorphism}, it turns out that in this case all Jordan embeddings $\mathcal{A}_\rho \to M_n$ are in fact rank preservers, which will be a topic of the next section.
		\end{remark}

		\section{Rank preserver}\label{sec:rank preserver}
		
		\subsection{Rank preservers} As previously announced, in this section for an arbitrary SMA  $\mathcal{A}_\rho \subseteq M_n$ we fully describe the form of all rank preservers $\mathcal{A}_\rho \to M_n$ (Theorem \ref{thm:rank preserver}). We start with the following lemma.

		\begin{lemma}\label{le:extending transitive rank preserver}
			Let $\ca{A}_\rho \subseteq M_n$ be an SMA. Assume that $g : \rho \to \C^\times$ is a transitive map such that $$g|_{\rho \cap \{1,\ldots,n-1\}^2} \equiv 1$$ and that the induced automorphism $g^* : \ca{A}_\rho \to \ca{A}_\rho$ is a rank preserver. Then $g$ is trivial. 
		\end{lemma}
		
		To motivate the proof of Lemma \ref{le:extending transitive rank preserver}, which is rather technical, we first illustrate it on a concrete example.
		
		\begin{example}
			Consider the SMA $\ca{A}_\rho \subseteq M_{10}$ given by
			$$\ca{A}_\rho = \begin{bmatrix}
				(1,1) & (1,2) & 0 & 0 & 0 & 0 & 0 & 0 & 0 & (1,10)\\
				0 & (2,2) & 0 & 0 & 0 & 0 & 0 & 0 & 0 & 0 \\
				0 & (3,2) & (3,3) & (3,4) & 0 & 0 & 0 & 0 & 0 & 0 \\
				0 & 0 & 0 & (4,4) & 0 & 0 & 0 & 0 & 0 & 0 \\
				0 & 0 & 0 & (5,4) & (5,5) & (5,6) & 0 & 0 & 0 & 0 \\
				0 & 0 & 0 & 0 & 0 & (6,6) & 0 & 0 & 0 & 0 \\
				0 & 0 & 0 & 0 & 0 & (7,6) & (7,7) & (7,8) & 0 & 0 \\
				0 & 0 & 0 & 0 & 0 & 0 & 0 & (8,8) & 0 & 0 \\
				0 & 0 & 0 & 0 & 0 & 0 & 0 & (9,8) & (9,9) & (9,10) \\
				0 & 0 & 0 & 0 & 0 & 0 & 0 & 0 & 0 & (10,10)\\
			\end{bmatrix}.$$
			
			If $\tripprox$ denotes the equivalence relation on $\{1,\ldots,9\}$ with respect to the quasi-order $\rho \cap \{1,\ldots,9\}^2$, then by Lemma \ref{le:extending g trivially} we have to show the implications
			$$\begin{cases}
				(\forall i,j \in (\rho^\times)^{-1}(10))(i \tripprox j  \implies g(i,10) = g(j,10)),\\
				(\forall i,j \in (\rho^\times)(10))(i \tripprox j  \implies g(10,i) = g(10,j)).
			\end{cases}$$
			The second implication is vacuously true, and to prove the first one we only have to consider $1, 9 \in (\rho^\times)^{-1}(10)$. Their equivalence $1 \tripprox 9$ (in $\rho \cap \{1,\ldots,9\}^2$) manifests through the chain of length $8$ given by
			$$(1,2),(3,2),(3,4),(5,4),(5,6),(7,6),(7,8),(9,8).$$
			Consider the following matrix $A \in \ca{A}_\rho$, which has $\pm 1$ at those exact positions:
			$$A = \begin{bmatrix}
				0 & 1 & 0 & 0 & 0 & 0 & 0 & 0 & 0 & 1\\
				0 & 0 & 0 & 0 & 0 & 0 & 0 & 0 & 0 & 0 \\
				0 & 1 & 0 & -1 & 0 & 0 & 0 & 0 & 0 & 0 \\
				0 & 0 & 0 & 0 & 0 & 0 & 0 & 0 & 0 & 0 \\
				0 & 0 & 0 & -1 & 0 & 1 & 0 & 0 & 0 & 0 \\
				0 & 0 & 0 & 0 & 0 & 0 & 0 & 0 & 0 & 0 \\
				0 & 0 & 0 & 0 & 0 & 1 & 0 & -1 & 0 & 0 \\
				0 & 0 & 0 & 0 & 0 & 0 & 0 & 0 & 0 & 0 \\
				0 & 0 & 0 & 0 & 0 & 0 & 0 & -1 & 0 & -1\\
				0 & 0 & 0 & 0 & 0 & 0 & 0 & 0 & 0 & 0
			\end{bmatrix}.$$
			Clearly, the first four nonzero columns of $A$ are linearly independent, while the last one is precisely the sum of all the previous ones. Therefore, $r(A) = 4$ and hence the rank of
			$$g^*(A) = \begin{bmatrix}
				0 & 1 & 0 & 0 & 0 & 0 & 0 & 0 & 0 & g(1,10)\\
				0 & 0 & 0 & 0 & 0 & 0 & 0 & 0 & 0 & 0 \\
				0 & 1 & 0 & -1 & 0 & 0 & 0 & 0 & 0 & 0 \\
				0 & 0 & 0 & 0 & 0 & 0 & 0 & 0 & 0 & 0 \\
				0 & 0 & 0 & -1 & 0 & 1 & 0 & 0 & 0 & 0 \\
				0 & 0 & 0 & 0 & 0 & 0 & 0 & 0 & 0 & 0 \\
				0 & 0 & 0 & 0 & 0 & 1 & 0 & -1 & 0 & 0 \\
				0 & 0 & 0 & 0 & 0 & 0 & 0 & 0 & 0 & 0 \\
				0 & 0 & 0 & 0 & 0 & 0 & 0 & -1 & 0 & -g(9,10)\\
				0 & 0 & 0 & 0 & 0 & 0 & 0 & 0 & 0 & 0
			\end{bmatrix}$$
			is $4$ as well. Now it is easy to arrive at $g(1,10) = g(9,10)$, which is the desired conclusion.\smallskip
		\end{example}
		
		In general, every pair $a \tripprox b$ for which the implications from Lemma \ref{le:extending g trivially} apply requires us to produce a similar chain of consecutive positions which establishes the equivalence. Hence, first we prove this auxiliary lemma:
		
		\begin{lemma}\label{le:chain of alternating pairs}
			Suppose $\rho$ is a quasi-order on $\{1,\ldots,n\}$ and assume that distinct $a,b \in \{1,\ldots,n\}$ satisfy $a \tripprox b$ (where $\tripprox$ corresponds to $\rho$). Then there exist $m \in \N$ and distinct $a = i_0, i_1,\ldots,i_{m-1},i_m = b \in \{1,\ldots,n\}$ such that at least one of the following is true:
			\begin{enumerate}
				\item $(a,i_1),(i_2,i_1), (i_2,i_3),(i_4,i_3),\ldots,(i_{m-1},i_{m-2}),(i_{m-1},b) \in \rho^\times$,
				\item $(a,i_1),(i_2,i_1), (i_2,i_3),(i_4,i_3),\ldots,(i_{m-2},i_{m-1}),(b,i_{m-1}) \in \rho^\times$,
				\item $(i_1,a),(i_1,i_2), (i_3,i_2),(i_3,i_4),\ldots,(i_{m-1},i_{m-2}),(i_{m-1},b) \in \rho^\times$,
				\item $(i_1,a),(i_1,i_2), (i_3,i_2),(i_3,i_4),\ldots,(i_{m-2},i_{m-1}),(b,i_{m-1}) \in \rho^\times$.
			\end{enumerate}
		\end{lemma}
		\begin{proof}
			By definition of $\tripprox$, there exists $m \in \N$ and $a = i_0, i_1,\ldots,i_{m-1},i_m = b \in \{1,\ldots,n\}$ such that
			\begin{equation*}
				a = i_0 \tripprox_0 i_1 \tripprox_0 i_2 \tripprox_0 \cdots \tripprox i_{m-1} \tripprox_0 i_m = b.
			\end{equation*}
			We further assume that $m$ is chosen to be minimal (in particular, the elements are distinct). Certainly $m \ge 1$ since $a \ne b$ so $i_1$ exists. \smallskip
			
			By definition of $a \tripprox_0 i_1$, we have $(a,i_1) \in \rho$ or $(i_1,a)\in \rho$. First we assume the former. Define a sequence $(s_j)_{j=1}^m$ of length $m$ in $\{1,\ldots,n\}^2$ by
			$$s_j = \begin{cases}
				(i_{j-1},i_j), &\quad\text{ if $j$ is odd},\\
				(i_{j},i_{j-1}), &\quad\text{ if $j$ is even}.
			\end{cases}$$
			Note that $(s_j)_{j=1}^m$ is precisely a sequence of the form $(1)$ or $(2)$ if $m$ is odd or even, respectively. It therefore suffices to prove that $s_j \in \rho$ for all $1 \le j \le m$. We prove this fact by induction on $j$. For $j=1$ we have $s_1 = (a,i_1) \in \rho$ by assumption. Assume that $m \ge 2$ and that $s_j \in \rho$ for some $1 \le j \le m-1$. We aim to prove $s_{j+1} \in \rho$. If $j$ is even, then by the inductive hypothesis, we have $s_j = (i_{j},i_{j-1}) \in \rho$. Since $i_{j} \tripprox_0 i_{j+1}$, by definition we have $(i_j,i_{j+1}) \in \rho$ or $(i_{j+1},i_{j}) \in \rho$. If the latter is true, by transitivity we would conclude $(i_{j+1},i_{j-1}) \in \rho$, and hence $$a = i_0 \tripprox_0 \cdots \tripprox_0 i_{j-1} \tripprox_0 i_{j+1} \tripprox_0 \cdots \tripprox_0 i_m = b$$
			which contradicts the minimality of $m$. Therefore, $s_{j+1} = (i_j,i_{j+1}) \in \rho$, which is what we wanted to show. On the other hand, if $j$ is odd, then by the inductive hypothesis we have $s_j = (i_{j-1},i_j) \in \rho$. Since $i_{j} \tripprox_0 i_{j+1}$, by definition we have $(i_j,i_{j+1}) \in \rho$ or $(i_{j+1},i_{j}) \in \rho$, but the former would, by transitivity, again contradict the minimality of $m$. Therefore, we conclude $s_{j+1} = (i_{j+1},i_{j}) \in \rho$, thus completing the inductive step.
			
			The other case $(i_1,a)\in \rho$ follows from the above argument applied on the reverse quasi-order $\rho^t$, yielding a sequence of the form $(3)$ or $(4)$.
		\end{proof}
		
		\begin{proof}[Proof of Lemma \ref{le:extending transitive rank preserver}]
			If $\tripprox$ denotes the equivalence relation on $\{1,\ldots,n-1\}$ with respect to the quasi-order $\rho \cap \{1,\ldots,n-1\}^2$, then by Lemma \ref{le:extending g trivially} we have to show the implications
			$$\begin{cases}
				(\forall i,j \in  (\rho^\times)^{-1}(n))(i \tripprox j  \implies g(i,n) = g(j,n)),\\
				(\forall i,j \in  (\rho^\times)(n))(i \tripprox j  \implies g(n,i) = g(n,j)).
			\end{cases}$$
			We will show the first one as the second one will then follow by considering the reverse quasi-order ${\rho^t}$ and the corresponding algebra $\ca{A}_{\rho^t}$.\smallskip
			
			Suppose that distinct $a,b \in (\rho^\times)^{-1}(n)$ satisfy $a \tripprox b$. Then by Lemma \ref{le:chain of alternating pairs} applied on the quasi-order $\rho \cap \{1,\ldots,9\}^2$ there exists $m \in \N$ and distinct $a = i_0, i_1,\ldots,i_{m-1},i_m = b \in \{1,\ldots,n-1\}$ such that at least one of the conditions $(1)-(4)$ is true.\smallskip
			
			Assume first that $m=1$. Then $(a,b) \in \rho$ or $(b, a) \in \rho$. In the first case we notice that $(a,b), (a,n), (b,b), (b,n)$ form a rectangle in $\ca{A}_\rho$ so by Lemma \ref{le:rank preserver rectangles} it follows
			$$0 = \begin{vmatrix}
				g(a,b) & g(a,n) \\ g(b,b) & g(b,n)
			\end{vmatrix} = \begin{vmatrix}
				1 & g(a,n) \\ 1 & g(b,n)
			\end{vmatrix} = g(b,n) - g(a,n).$$
			In the second case we proceed similarly with the rectangle $(b,a), (b,n), (a,a), (a,n)$. Therefore in the remainder of the proof we can assume that $m \ge 2$.
			
			\smallskip
			
			Assume that we are in the case $(2)$. In particular, $m$ is even. Consider the matrix $A \in \ca{A}_{\rho}$ given by
			$$A := (E_{i_0i_1} + E_{i_2i_1}) - (E_{i_2i_3} + E_{i_4i_3}) + \cdots + (-1)^{\frac{m}2-1} (E_{i_{m-2}i_{m-1}}+E_{i_m i_{m-1}}) + E_{an}+(-1)^{\frac{m}2-1}E_{bn}.$$
			Note that $A$ has rank exactly equal to $\frac{m}2$. Namely, the $\frac{m}2$ columns $i_1,i_3,\ldots,i_{m-1}$ are linearly independent, while its $n$-th column is their sum:
			$$\sum_{\substack{0 \le j \le m-2\\j \text{ even}}} \underbrace{(-1)^{\frac{j}2} (e_{i_j} + e_{i_{j+2}})}_{\text{$i_{j+1}$-th column of $A$}} = \underbrace{e_{a} + (-1)^{\frac{m}2-1} e_{b}}_{\text{$n$-th column of $A$}}.$$
			The map $g^*$ maps the matrix $A$ to the matrix
			\begin{align*}
				g^*(A) &= (E_{i_0i_1} + E_{i_2i_1}) - (E_{i_2i_3} + E_{i_4i_3}) + \cdots + (-1)^{\frac{m}2-1} (E_{i_{m-2}i_{m-1}}+E_{i_m i_{m-1}}) \\
				&\quad+ g(a,n)E_{an}+(-1)^{\frac{m}2-1}g(b,n)E_{bn}
			\end{align*}
			which, by assumption, also has rank $\frac{m}2$. Hence there exist scalars $\alpha_0,\alpha_2,\ldots,\alpha_{m-2} \in \C$ such that
			$$\sum_{\substack{0 \le j \le m-2\\j \text{ even}}} \alpha_j \underbrace{(-1)^{\frac{j}2} (e_{i_j} + e_{i_{j+2}})}_{\text{$i_{j+1}$-th column of $g^*(A)$}} =  \underbrace{g(a,n)e_{a} + (-1)^{\frac{m}2-1}g(b,n) e_{b}}_{\text{$n$-th column of $g^*(A)$}}.$$
			Comparing coefficients of $e_{i_j}$ yields
			$$\begin{cases}
				i_0 = a:&\qquad \alpha_0 = g(a,n),\\
				i_j \text{ for even } 2 \le j \le m-2:&\qquad (-1)^{\frac{j-2}2}\alpha_{j-2} + (-1)^{\frac{j}2} \alpha_{j} = 0 \implies \alpha_j = \alpha_{j-2},\\
				i_m = b:&\qquad (-1)^{\frac{m-2}2}\alpha_{m-2} = (-1)^{\frac{m}2-1} g(b,n) \implies \alpha_{m-2} = g(b,n).\\
			\end{cases}$$
			Inductively it follows
			$$g(a,n) = \alpha_0 = \alpha_2 = \cdots = \alpha_{m-2} = g(b,n),$$
			which is precisely what we desired.
			
			\smallskip
			
			Assume that we are in the case $(1)$. Then by transitivity $(i_{m-1},b), (b, n) \in \rho^\times$ implies $(i_{m-1},n) \in \rho^\times$. Then $a$ and $i_{m-1}$ are connected by a sequence of the form $(2)$ and hence $g(a,n) = g(i_{m-1},n)$. Furthermore, we have $i_{m-1} \tripprox_0 b$ so by the $m=1$ case it follows $g(i_{m-1},n) = g(b,n)$.\smallskip
			
			Cases $(3)$ and $(4)$ are treated similarly.
			
			
		\end{proof}
		
		\begin{lemma}\label{le:transitive rank preserver is trivial}
			Let $\ca{A}_\rho \subseteq M_n$ be an SMA and let $g : \rho \to \C^\times$ be a transitive map. Suppose that the induced automorphism $g^* : \ca{A}_\rho \to \ca{A}_\rho$ is a rank preserver. Then $g$ is trivial. 
		\end{lemma}
		\begin{proof}
			We prove the claim by induction on $n$. For $n=1$ the claim is clear. Suppose that $n \ge 2$ and that the claim holds for all SMAs contained in $M_{n-1}$. As the automorphism 
			$$(g|_{\rho \cap \{1,\ldots,n-1\}^2})^*: \ca{A}_{\rho \cap \{1,\ldots,n-1\}^2} \to \ca{A}_{\rho \cap \{1,\ldots,n-1\}^2}$$
			induced by the transitive map $g|_{\rho \cap \{1,\ldots,n-1\}^2}$ coincides with the restriction of $g^*$ on $\ca{A}_{\rho \cap \{1,\ldots,n-1\}^2}$, which is a rank preserver, by the induction hypothesis $g|_{\rho \cap \{1,\ldots,n-1\}^2}$
			is trivial. Hence, there exists a map $s : \{1,\ldots,n-1\} \to \C^\times$ such that $$g(i,j) = \frac{s(i)}{s(j)}, \qquad \text{ for all } (i,j) \in \rho \cap \{1,\ldots,n-1\}^2.$$
			Additionally set $s(n) := 1$ and define a new transitive map $$h : \rho \to \C^\times, \qquad h(i,j) := \frac{s(j)}{s(i)}g(i,j).$$
			Note that $h|_{\rho \cap \{1,\ldots,n-1\}^2} \equiv 1$. Furthermore, if we denote $D := \diag(s(1),\ldots,s(n)) \in \ca{D}_n^\times$, then one easily verifies that
			$$h^* = D^{-1}g^*(\cdot)D$$
			so $h^*$ is also a rank-one preserver. Lemma \ref{le:extending transitive rank preserver} implies that $h$ is trivial, which implies that $g$ is trivial as well.
		\end{proof}

		\begin{theorem}\label{thm:rank preserver}
			Let $\ca{A}_\rho \subseteq M_n$ be an SMA. A map $\phi : \ca{A}_\rho \to M_n$ is a linear unital rank preserver if and only if there exists an invertible matrix $T \in M_n^\times$ and a central idempotent $P \in Z(\ca{A}_\rho)$ such that
			\begin{equation}\label{eq:rank preserver}
				\phi(\cdot) = T\left(P(\cdot) + (I-P)(\cdot)^t\right)T^{-1}.
			\end{equation}
		\end{theorem}
		\begin{proof}
			\fbox{$\implies$} Suppose that $\phi : \ca{A}_\rho \to M_n$ is a linear unital rank preserver. By Theorem \ref{thm:rank-one preserver} (a) it follows that $\phi$ is a Jordan embedding. Theorem \ref{thm:general form of phi} then implies that there exists an invertible matrix $T \in M_n^\times$, a central idempotent $P \in Z(\ca{A}_\rho)$, and a transitive map $g : \rho \to \C^\times$ such that
			$$\phi(\cdot) = S(Pg^*(\cdot) + (I-P)g^*(\cdot)^t)S^{-1}.$$
			By a similar argument as in the proof of Theorem \ref{thm:rank-one preserver} (b) we obtain that $r(X)= r(\phi(X)) = r(g^*(X))$ for each $X \in \ca{A}_{\rho}$.
			It follows that $g^*$ is a rank preserver, so by Lemma \ref{le:transitive rank preserver is trivial} we conclude that the transitive map $g$ is trivial. By Lemma \ref{le:g trivial iff induced map is inner automorphism}, the induced automorphism $g^*$ is of the form $g^* = D(\cdot)D^{-1}$ for some $D \in \ca{D}_n^\times$. Consider $\Gamma \in \ca{D}_n^\times$ given by
			$$\Gamma_{jj}:= \begin{cases}
				D_{jj}, &\text{ if $(j,j) \in \supp P$},\\
				\frac1{D_{jj}}, &\text{ if $(j,j) \notin \supp P$}.
			\end{cases}$$
			For all $X \in \ca{A}_\rho$ we have
			\begin{align*}
				\phi(X) &= S\left(Pg^*(X) + (I-P)g^*(X)^t\right)S^{-1} \\
				&= S\left(PDXD^{-1} + (I-P)D^{-1}X^tD\right)S^{-1} \\
				&=  S\left(D(PX)D^{-1} + D^{-1}((I-P)X^t)D\right)S^{-1}\\
				&=  S\left(\Gamma (PX)\Gamma^{-1} +\Gamma((I-P)X^t)\Gamma^{-1}\right)S^{-1}\\
				&= T (PX + (I-P)X^t) T^{-1}
			\end{align*}
			where $T := S\Gamma \in \ca{A}_\rho^\times$.
			
			\smallskip
			
			\noindent \fbox{$\impliedby$} Suppose that $\phi : \ca{A}_\rho \to M_n$ is of the form \eqref{eq:rank preserver} for some invertible matrix $T \in M_n^\times$ and a central idempotent $P \in Z(\ca{A}_\rho)$. Then, by  Lemma \ref{le:rank of canonical Jordan homomorphism}, for all $X \in \ca{A}_\rho$ we have
			\begin{align*}
				r(\phi(X)) &= r(PX + (I-P)X^t) = r(X),
			\end{align*}
			so that $\phi$ is a rank preserver.
			
		\end{proof}
		
		\begin{remark}\phantom{x}
			\begin{enumerate}[(a)]
				\item In Lemmas \ref{le:extending transitive rank preserver} and \ref{le:transitive rank preserver is trivial} and in Theorem \ref{thm:rank preserver} it actually suffices to assume that the respective map preserves the rank of matrices up to $\max\left\{1,\left\lfloor \frac{n}2 \right\rfloor-1 \right\}$. Indeed, let $1 \le a,b \le n-1$ and $m \in \N$ be as in the beginning of the proof of Lemma \ref{le:extending transitive rank preserver}. To prove the result, we use the fact that $g^*$ is a rank-one preserver, while if $m \ge 2$ then we additionally use that $g^*$ preserves rank $\frac{m}2$. Since $i_0,i_1,\ldots, i_m \in \{1,\ldots, n-1\}$ are distinct, it follows that $m+1 \le n-1$ so we arrive at $\frac{m}2 \le \frac{n}2 - 1$.
				\item As a consequence of (a), every unital rank-one preserver $\phi : \ca{A}_\rho \to M_n$ where $n \le 5$, is necessarily a rank preserver. Moreover, for general $n \in \N$ and $\ca{A}_\rho \subseteq M_n$, using the central decomposition of $\mathcal{A}_\rho$ (Remark \ref{re:central decomposition}), it is not difficult to show that the same result holds whenever $\abs{C} \le 5$ for all $C \in \ca{Q}$. 
			\end{enumerate}
		\end{remark}
		
		\begin{corollary}\label{cor:rank preserver without unitality}
			Let $\ca{A}_\rho \subseteq M_n$ be an SMA. A map $\phi : \ca{A}_\rho \to M_n$ is a linear rank preserver if and only if there exist invertible matrices $S,T \in M_n^\times$ and a central idempotent $P \in Z(\ca{A}_\rho)$ such that
			$$\phi(\cdot) = S\left(P(\cdot) + (I-P)(\cdot)^t\right)T.$$      
		\end{corollary}
		\begin{proof}
			In either direction it follows that the matrix $\phi(I) \in M_n$ is invertible. The equivalence follows by applying Theorem \ref{thm:rank preserver} to the unital map $\phi(I)^{-1}\phi(\cdot)$.
		\end{proof}
		
		\begin{example}
			Corollary \ref{cor:rank preserver without unitality} cannot be further strengthened  by assuming  that $\phi : \ca{A}_\rho \to M_n$ is only a rank-$k$ preserver for all $1 \le k \le n-1$. Indeed, let $\ca{A}_\rho = \ca{D}_n$ and consider the  map $\phi:\ca{D}_n \to M_n$ defined by
			$$\diag(x_{11},\ldots,x_{nn}) \mapsto \begin{bmatrix}
				x_{11} + x_{nn} & x_{nn} & \cdots & x_{nn} & 0\\
				x_{nn} & x_{22} + x_{nn} & \cdots & x_{nn} & 0\\
				\vdots & \vdots & \ddots & \vdots & \vdots \\
				x_{nn} & x_{nn} & \cdots & x_{(n-1)(n-1)} + x_{nn} & 0\\
				0 & 0 & \cdots & 0 & 0
			\end{bmatrix}.$$
			Clearly, $\phi(I)$ is a singular matrix, so in particular $\phi$ is not of the form outlined in Corollary \ref{cor:rank preserver without unitality}. Further, $\phi$ preserves rank of all singular matrices (hence ranks up to $n-1$). Indeed, let $X \in \mathcal{D}_n$ be a singular matrix. We need to prove that $r(\phi(X))=r(X)$. This is clear if $x_{nn} = 0$ so assume that $x_{nn} \ne 0$. By conjugating by a permutation matrix we can further assume that $x_{11} = 0$. By subtracting the first row from rows $2,\ldots,n-1$, we conclude $\phi(X)$ has the same rank as
			$$\begin{bmatrix}
				x_{nn} & x_{nn} & \cdots & x_{nn} & 0\\
				0 & x_{22} & \cdots & 0 & 0\\
				\vdots & \vdots & \ddots & \vdots & \vdots \\
				0& 0 & \cdots & x_{(n-1)(n-1)} & 0\\
				0 & 0 & \cdots & 0 & 0
			\end{bmatrix} \sim \diag(x_{nn},x_{22}, \ldots x_{(n-1)(n-1)},0)\sim X,$$
			where $\sim$ denotes the matrix equivalence.
		\end{example}

	\end{document}